\documentclass[11pt,a4paper]{amsart}
\usepackage{amssymb,amsthm}

\usepackage[truedimen,margin=30truemm]{geometry}

\usepackage[pdftex,colorlinks=true,bookmarks=true,citecolor=blue,linkcolor=blue,pdfauthor={},pdftitle={}]{hyperref}

\theoremstyle{plain}
\newtheorem{theorem}{Theorem}[section]
\newtheorem{lemma}[theorem]{Lemma}
\newtheorem{proposition}[theorem]{Proposition}

\theoremstyle{definition}
\newtheorem{definition}[theorem]{Definition}
\newtheorem{assumption}[theorem]{Assumption}

\theoremstyle{remark}
\newtheorem{remark}[theorem]{Remark}

\numberwithin{equation}{section}

\begin{document}

\title[Stationary Navier--Stokes equations in a curved thin domain]{Approximation of a solution to the stationary Navier--Stokes equations in a curved thin domain by a solution to thin-film limit equations}

\author[T.-H. Miura]{Tatsu-Hiko Miura}
\address{Graduate School of Science and Technology, Hirosaki University, 3, Bunkyo-cho, Hirosaki-shi, Aomori, 036-8561, Japan}
\email{thmiura623@hirosaki-u.ac.jp}



\begin{abstract}
  We consider the stationary Navier--Stokes equations in a three-dimensional curved thin domain around a given closed surface under the slip boundary conditions.
  Our aim is to show that a solution to the bulk equations is approximated by a solution to limit equations on the surface appearing in the thin-film limit of the bulk equations.
  To this end, we take the average of the bulk solution in the thin direction and estimate the difference of the averaged bulk solution and the surface solution.
  Then we combine an obtained difference estimate on the surface with an estimate for the difference of the bulk solution and its average to get a difference estimate for the bulk and surface solutions in the thin domain, which shows that the bulk solution is approximated by the surface one when the thickness of the thin domain is sufficiently small.
\end{abstract}

\maketitle

\section{Introduction} \label{S:Intro}

\subsection{Problem settings and the outline of main results} \label{SS:In_Pro}
Let $\Gamma$ be a closed surface in $\mathbb{R}^3$ with unit outward normal vector field $n$.
Also, let $g_0$ and $g_1$ be functions on $\Gamma$ such that $g=g_1-g_0\geq c$ on $\Gamma$ with some constant $c>0$.
For a sufficiently small $\varepsilon>0$, we define a curved thin domain $\Omega_\varepsilon$ in $\mathbb{R}^3$ and its inner and outer boundaries $\Gamma_\varepsilon^0$ and $\Gamma_\varepsilon^1$ by
\begin{align} \label{E:Def_CTD}
  \begin{aligned}
    \Omega_\varepsilon &= \{y+rn(y) \mid y\in\Gamma, \, \varepsilon g_0(y)<r<\varepsilon g_1(y)\}, \\
    \Gamma_\varepsilon^i &= \{y+\varepsilon g_i(y)n(y) \mid y\in\Gamma\}, \quad i=0,1
  \end{aligned}
\end{align}
and set the whole boundary $\Gamma_\varepsilon=\Gamma_\varepsilon^0\cup\Gamma_\varepsilon^1$ of $\Omega_\varepsilon$.
We consider the stationary Navier--Stokes equations under Navier's slip boundary conditions
\begin{align} \label{E:SNS_CTD}
  \left\{
  \begin{alignedat}{3}
    -\nu\Delta u^\varepsilon+(u^\varepsilon\cdot\nabla)u^\varepsilon+\nabla p^\varepsilon = f^\varepsilon, \quad \mathrm{div}\,u^\varepsilon &= 0 &\quad &\text{in} &\quad &\Omega_\varepsilon, \\
    u^\varepsilon\cdot n_\varepsilon = 0, \quad 2\nu P_\varepsilon D(u^\varepsilon)n_\varepsilon+\gamma_\varepsilon u^\varepsilon &= 0 &\quad &\text{on} &\quad &\Gamma_\varepsilon.
  \end{alignedat}
  \right.
\end{align}
Here $\nu>0$ is the viscosity coefficient independent of $\varepsilon$ and $f^\varepsilon$ is a given external force depending on $\varepsilon$.
Also, $n_\varepsilon$ is the unit outward normal vector field of $\Gamma_\varepsilon$ and
\begin{align*}
  D(u^\varepsilon) = \frac{\nabla u^\varepsilon+(\nabla u^\varepsilon)^T}{2}, \quad P_\varepsilon = I_3-n_\varepsilon\otimes n_\varepsilon
\end{align*}
are the strain rate tensor and the orthogonal projection onto the tangent plane of $\Gamma_\varepsilon$, where $I_3$ and $n_\varepsilon\otimes n_\varepsilon$ are the $3\times 3$ identity matrix and the tensor product of $n_\varepsilon$ with itself.
The symbol $\gamma_\varepsilon$ stands for the friction coefficient and is given by
\begin{align} \label{E:Fric}
  \gamma_\varepsilon|_{\Gamma_\varepsilon^i} = \gamma_\varepsilon^i \quad\text{on}\quad \Gamma_\varepsilon^i, \quad i=0,1,
\end{align}
where $\gamma_\varepsilon^0$ and $\gamma_\varepsilon^1$ are given nonnegative constants depending on $\varepsilon$.

Since $\Omega_\varepsilon$ shrinks to $\Gamma$ as $\varepsilon\to0$, one naturally expects that a solution to \eqref{E:SNS_CTD} is close to some function on $\Gamma$ when $\varepsilon$ is small.
The purpose of this paper is to show that a solution to \eqref{E:SNS_CTD} is approximated by a solution to
\begin{align} \label{E:Lim_Eq}
  \left\{
  \begin{aligned}
    &-2\nu\left\{P\mathrm{div}_\Gamma[gD_\Gamma(v)]-\frac{1}{g}(v\cdot\nabla_\Gamma g)\nabla_\Gamma g\right\} \\
    &\qquad +(\gamma^0+\gamma^1)v+g\overline{\nabla}_vv+g\nabla_\Gamma q = gf \quad\text{on}\quad \Gamma, \\
    &\mathrm{div}_\Gamma(gv) = 0 \quad\text{on}\quad \Gamma,
  \end{aligned}
  \right.
\end{align}
which we call the limit equations of \eqref{E:SNS_CTD}.
Here unknown functions are the tangential velocity field $v$ and the pressure $q$.
Also, $f$ is a given external force.
We write $P$, $\nabla_\Gamma$, $\mathrm{div}_\Gamma$, $D_\Gamma(v)$, and $\overline{\nabla}_vv$ for the orthogonal projection onto the tangent plane of $\Gamma$, the tangential gradient, the surface divergence, the surface strain rate tensor, and the covariant derivative of $v$ along itself, respectively.
Also, $\gamma^0$ and $\gamma^1$ are nonnegative constants which stands for the friction coefficients.
For details of notations, see Section \ref{S:Main}.
Note that, when $g\equiv 1$ on $\Gamma$ and $\gamma^0=\gamma^1=0$, the limit equations \eqref{E:Lim_Eq} reduce to the surface Navier--Stokes equations with Boussinesq--Scriven surface stress tensor (see \cite{Bou14,Scr60,Ari89})
\begin{align} \label{E:SNS_Surf}
  \left\{
  \begin{alignedat}{3}
    -2\nu P\mathrm{div}_\Gamma[D_\Gamma(v)]+\overline{\nabla}_vv+\nabla_\Gamma q &= f &\quad &\text{on} &\quad \Gamma, \\
    \mathrm{div}_\Gamma v &= 0 &\quad &\text{on} &\quad \Gamma.
  \end{alignedat}
  \right.
\end{align}
Moreover, the equations \eqref{E:SNS_Surf} are equivalent to the Navier--Stokes equations on an abstract Riemannian manifold (see \cite{EbiMar70,Tay92,ChCzDi17})
\begin{align} \label{E:SNS_Mani}
  \left\{
  \begin{alignedat}{3}
    -\nu\{\Delta_Bv+\mathrm{Ric}(v)\}+\overline{\nabla}_vv+\nabla_\Gamma q &= f &\quad &\text{on} &\quad \Gamma, \\
    \mathrm{div}_\Gamma v &= 0 &\quad &\text{on} &\quad \Gamma,
  \end{alignedat}
  \right.
\end{align}
where $\Delta_B$ is the Bochner Laplacian on $\Gamma$ and $\mathrm{Ric}$ is the Ricci curvature of $\Gamma$ (see e.g. \cite[Lemma C.11]{Miu20_03} for the equivalence of the above equations).

In the nonstationary setting, we rigorously derived the limit equations \eqref{E:Lim_Eq} from the bulk equations \eqref{E:SNS_CTD} by the thin-film limit in our previous work \cite{Miu20_03}.
There we proved under suitable assumptions that, for an $L^2$-strong solution $u^\varepsilon$ to the nonstationary Navier--Stokes equations in $\Omega_\varepsilon$, its average
\begin{align*}
  Mu^\varepsilon(y) = \frac{1}{\varepsilon g(y)}\int_{\varepsilon g_0(y)}^{\varepsilon g_1(y)}u^\varepsilon(y+rn(y))\,dr, \quad y\in\Gamma
\end{align*}
converges weakly to a tangential vector field $v$ on $\Gamma$ in an appropriate function space as $\varepsilon\to0$, and derived the nonstationary limit equations on $\Gamma$ by characterizing $v$ as a unique $L^2$-weak solution to the limit equations.
We also obtained some estimates for the difference of $u^\varepsilon$ and $v$ which show that $v$ approximates $u^\varepsilon$ in the $L^2$ sense when $\varepsilon$ is small.

As in the nonstationary case \cite{Miu20_03}, we can derive \eqref{E:Lim_Eq} from \eqref{E:SNS_CTD} by means of convergence of a solution and characterization of the limit, but the procedure is the same so we omit it here.
In this paper, we focus on difference estimates for the solutions $u^\varepsilon$ to \eqref{E:SNS_CTD} and $v$ to \eqref{E:Lim_Eq}.
Let us fix some notations and formally state our main results (see Section \ref{S:Main} for details).
Let $\mathbb{P}_\varepsilon$ be the orthogonal projection from $L^2(\Omega_\varepsilon)^3$ onto a function space $\mathcal{H}_\varepsilon$ given in \eqref{E:Def_Heps}, which is the standard $L^2$ solenoidal space on $\Omega_\varepsilon$ or its subspace.
For a vector field $u$ on $\Omega_\varepsilon$, let $M_\tau u$ be the tangential component of the average $Mu$ on $\Gamma$.
Let
\begin{align*}
  H^1(\Gamma,T\Gamma) = \{v\in H^1(\Gamma)^3 \mid \text{$v\cdot n=0$ on $\Gamma$}\}
\end{align*}
and $H^{-1}(\Gamma,T\Gamma)$ be the dual space of $H^1(\Gamma,T\Gamma)$.
Formally speaking, our main results are as follows (see Theorems \ref{T:Error} and \ref{T:Er_CE} for the precise statements).

\begin{theorem} \label{T:Er_Intro}
  Let $f^\varepsilon\in L^2(\Omega_\varepsilon)^3$, $f\in H^{-1}(\Gamma,T\Gamma)$, and $u^\varepsilon$ and $v$ be weak solutions to \eqref{E:SNS_CTD} and \eqref{E:Lim_Eq}, respectively.
  Under suitable assumptions, suppose that there exist $c_1,c_2>0$ and $\alpha\in(0,1]$ independent of $\varepsilon$ such that
  \begin{align} \label{E:F_Intro}
    \|\mathbb{P}_\varepsilon f^\varepsilon\|_{L^2(\Omega_\varepsilon)}^2 \leq c_1\varepsilon^{-1+\alpha}, \quad \|M_\tau\mathbb{P}_\varepsilon f^\varepsilon\|_{H^{-1}(\Gamma,T\Gamma)}^2 \leq c_2
  \end{align}
  for all $\varepsilon>0$ sufficiently small.
  Then there exist $c,\rho>0$ independent of $\varepsilon$ such that
  \begin{align} \label{E:Er_Intro}
    \|M_\tau u^\varepsilon-v\|_{H^1(\Gamma)} \leq c\Bigl(\delta(\varepsilon)+\|M_\tau\mathbb{P}_\varepsilon f^\varepsilon-f\|_{H^{-1}(\Gamma,T\Gamma)}\Bigr)
  \end{align}
  for all $\varepsilon>0$ sufficiently small provided that $\|v\|_{H^1(\Gamma)}\leq\rho$, where
  \begin{align*}
    \delta(\varepsilon) = \varepsilon^{\alpha/4}+\sum_{i=0,1}\left|\frac{\gamma_\varepsilon^i}{\varepsilon}-\gamma^i\right|.
  \end{align*}
  Moreover, we have the following difference estimate in $\Omega_\varepsilon$:
  \begin{align} \label{E:ECE_Intro}
    \varepsilon^{-1/2}\|u^\varepsilon-\bar{v}\|_{L^2(\Omega_\varepsilon)} \leq c\Bigl(\delta(\varepsilon)+\|M_\tau\mathbb{P}_\varepsilon f^\varepsilon-f\|_{H^{-1}(\Gamma,T\Gamma)}\Bigr).
  \end{align}
  Here $\bar{v}$ is the constant extension of $v$ in the normal direction of $\Gamma$.
\end{theorem}

Note that the left-hand side of \eqref{E:ECE_Intro} is divided by $\varepsilon^{1/2}$ since the $L^2(\Omega_\varepsilon)$-norm involves the square root of the thickness of $\Omega_\varepsilon$.
By \eqref{E:ECE_Intro}, we can say that the solution $v$ to the limit equations \eqref{E:Lim_Eq} approximates the solution $u^\varepsilon$ to the bulk equations \eqref{E:SNS_CTD} when $\varepsilon$ is small.
We also have an approximation result for $\nabla u^\varepsilon$ in $\Omega_\varepsilon$, but it involves the Weingarten map (the second fundamental form) of $\Gamma$.
For details, see Theorem \ref{T:Er_CE}.
We further note that the difference estimate \eqref{E:Er_Intro} also holds if we replace the tangential component $M_\tau u^\varepsilon$ by the average $Mu^\varepsilon$ itself (see Remark \ref{R:Er_Full}).

The above results are basically the same as in the nonstationary case \cite[Theorems 7.27 and 7.29]{Miu20_03}, but the additional condition $\|v\|_{H^1(\Gamma)}\leq\rho$ is required due to lack of Gronwall's inequality as in the proof of the uniqueness of a weak solution to the stationary Navier--Stokes equations.
In fact, the condition $\|v\|_{H^1(\Gamma)}\leq\rho$ also implies the uniqueness of a weak solution to the limit equations \eqref{E:Lim_Eq} in an appropriate class (see Remark \ref{R:Pf_TErr}).

\subsection{Idea of the proof} \label{SS:In_Idea}
The difference estimate \eqref{E:ECE_Intro} in $\Omega_\varepsilon$ is obtained just by combination of the difference estimate \eqref{E:Er_Intro} on $\Gamma$ and an estimate for the difference of $u^\varepsilon$ and $M_\tau u^\varepsilon$ (see Lemma \ref{L:Ave_Diff}).
Thus the main effort is devoted to the proof of \eqref{E:Er_Intro}.

To prove \eqref{E:Er_Intro}, we first take the average in the thin direction of a weak form of the bulk equations \eqref{E:SNS_CTD} and show that $M_\tau u^\varepsilon$ satisfies a weak form on $\Gamma$ which is close to a weak form of the limit equation \eqref{E:Lim_Eq} but involves some residual terms.
Then we show that the residual terms in the weak form of $M_\tau u^\varepsilon$ are sufficiently small, take the difference of the weak forms of $M_\tau u^\varepsilon$ and $v$, and test $M_\tau u^\varepsilon-v$ to obtain \eqref{E:Er_Intro}.
The first part was already done in the study of the nonstationary problem \cite{Miu20_03}, so we omit details in this paper.
Also, the idea of the second part is basically the same as in \cite{Miu20_03}, but we encounter some difficulties due to lack of Gronwall's inequality in the stationary problem.

To explain the difficulties, let us consider the simplest but important case:
\begin{itemize}
  \item $\Gamma=S^2$ is the unit sphere in $\mathbb{R}^3$,
  \item $g_0\equiv0$ and $g_1\equiv1$ (and thus $g=g_1-g_0\equiv1$) on $S^2$,
  \item $\gamma_\varepsilon^0=\gamma_\varepsilon^1=0$ for all $\varepsilon>0$ and $\gamma^0=\gamma^1=0$.
\end{itemize}
In this case, $\Omega_\varepsilon=\{x\in\mathbb{R}^3 \mid 1<|x|<1+\varepsilon\}$ is a thin spherical shell and the bulk equations \eqref{E:SNS_CTD} read as the Navier--Stokes equations with perfect slip boundary conditions
\begin{align} \label{E:Perfect}
  \left\{
  \begin{alignedat}{3}
    -\nu\Delta u^\varepsilon+(u^\varepsilon\cdot\nabla)u^\varepsilon+\nabla p^\varepsilon = f^\varepsilon, \quad \mathrm{div}\,u^\varepsilon &= 0 &\quad &\text{in} &\quad &\Omega_\varepsilon, \\
    u^\varepsilon\cdot n_\varepsilon = 0, \quad P_\varepsilon D(u^\varepsilon)n_\varepsilon &= 0 &\quad &\text{on} &\quad &\Gamma_\varepsilon.
  \end{alignedat}
  \right.
\end{align}
Also, the limit equations \eqref{E:Lim_Eq} reduce to the surface Navier--Stokes equations \eqref{E:SNS_Surf} on $S^2$.

As we mentioned above, the weak form on $\Gamma$ satisfied by $M_\tau u^\varepsilon$ is close to the one of the limit equations but involves some residual terms.
They are estimated by the $H^1(\Omega_\varepsilon)$- and $H^2(\Omega_\varepsilon)$-norms of $u^\varepsilon$ (see Lemmas \ref{L:BL_Appr} and \ref{L:TL_Appr}), so we need to estimate these norms explicitly in terms of $\varepsilon$ in order to show that the residual terms are sufficiently small as $\varepsilon\to0$.
To estimate the $H^1(\Omega_\varepsilon)$-norm of $u^\varepsilon$, we take $u^\varepsilon$ as a test function in a weak form of \eqref{E:Perfect}.
Then we have $a_\varepsilon(u^\varepsilon,u^\varepsilon)=(f^\varepsilon,u^\varepsilon)_{L^2(\Omega_\varepsilon)}$, where $a_\varepsilon$ is a bilinear form in the weak form of \eqref{E:Perfect} given by
\begin{align*}
  a_\varepsilon(u_1,u_2) = 2\nu\bigl(D(u_1),D(u_2)\bigr)_{L^2(\Omega_\varepsilon)}, \quad u_1,u_2\in H^1(\Omega_\varepsilon)^3.
\end{align*}
We estimate the inner product of $u^\varepsilon$ and $f^\varepsilon$ by using \eqref{E:F_Intro} and some auxiliary arguments.
Moreover, it was shown in our previous work \cite{Miu22_01} that the uniform coerciveness
\begin{align*}
  \|u^\varepsilon\|_{H^1(\Omega_\varepsilon)}^2 \leq ca_\varepsilon(u^\varepsilon,u^\varepsilon) = 2\nu c\|D(u^\varepsilon)\|_{L^2(\Omega_\varepsilon)}^2
\end{align*}
holds with constant $c>0$ independent of $\varepsilon$ if $u^\varepsilon$ satisfies
\begin{align} \label{E:OrR_Intro}
  (u^\varepsilon,w)_{L^2(\Omega_\varepsilon)} = 0 \quad\text{for all}\quad w\in\mathcal{R}, \quad \mathcal{R} = \{w(x)=a\times x, \, x\in\mathbb{R}^3 \mid a\in\mathbb{R}^3\}.
\end{align}
Hence, by assuming \eqref{E:OrR_Intro}, we can get an estimate for the $H^1(\Omega_\varepsilon)$-norm of $u^\varepsilon$ explicitly in terms of $\varepsilon$ (see Lemma \ref{L:NCW_H1} for details).
Note that, under the situation considered here, the uniform coerciveness of $a_\varepsilon$ is the uniform Korn inequality with constant independent of $\varepsilon$, and the condition \eqref{E:OrR_Intro} is the same as the one for a standard Korn inequality with constant possibly depending on $\varepsilon$.
In general, the validity of the uniform Korn inequality on a curved thin domain is closely related to Killing vector fields on a limit surface (see \cite{LewMul11,Miu22_01}).
The above situation is special since any Killing vector field on $S^2$ is of the form $w(y)=a\times y$ for $y\in S^2$.
In the case of a general curved thin domain and limit surface, we make suitable assumptions to confirm that the uniform coerciveness of the bilinear form $a_\varepsilon$ holds (see Assumptions \ref{Asmp_1} and \ref{Asmp_2}).
It is also necessary to estimate the $H^2(\Omega_\varepsilon)$-norm of $u^\varepsilon$ explicitly in terms of $\varepsilon$.
To this end, we take the $L^2(\Omega_\varepsilon)$-inner product of $A_\varepsilon u^\varepsilon$ with \eqref{E:Perfect}, where $A_\varepsilon$ is the Stokes operator on $\Omega_\varepsilon$, apply an explicit estimate for the inner product of $(u^\varepsilon\cdot\nabla)u^\varepsilon$ and $A_\varepsilon u^\varepsilon$ shown in \cite{Miu21_02}, and use the uniform equivalence of the norms $\|A_\varepsilon u^\varepsilon\|_{L^2(\Omega_\varepsilon)}$ and $\|u\|_{H^2(\Omega_\varepsilon)}$ given in \cite{Miu22_01}.
For details, see Lemma \ref{L:NCW_H2}.

Another and more crucial difficulty is in taking a test function in the difference of the weak forms of $M_\tau u^\varepsilon$ and $v$.
Recall that here we consider the simplest case where the limit equations are the surface Navier--Stokes equations \eqref{E:SNS_Surf} on $S^2$, so it is required that a test function $\eta$ of a weak form of \eqref{E:SNS_Surf} should satisfy $\mathrm{div}_\Gamma\eta=0$ on $S^2$.
However, since $\mathrm{div}_\Gamma(M_\tau u^\varepsilon)$ does not vanish in general, we cannot take $M_\tau u^\varepsilon-v$ as a test function.
To overcome this difficulty, we construct an auxiliary vector field $v_1^\varepsilon$ satisfying $\mathrm{div}_\Gamma v_1^\varepsilon=0$ by setting $v_1^\varepsilon=M_\tau u^\varepsilon-\nabla_\Gamma q^\varepsilon$, where $q^\varepsilon$ is a solution to an appropriate elliptic problem on $S^2$ involving $\mathrm{div}_\Gamma(M_\tau u^\varepsilon)$.
Then, using an estimate for $\mathrm{div}_\Gamma(M_\tau u^\varepsilon)$, we show that $v_1^\varepsilon$ is sufficiently close to $M_\tau u^\varepsilon$ as $\varepsilon\to0$.
Hence one would obtain the difference estimate \eqref{E:Er_Intro} if one could get a difference estimate for $v_1^\varepsilon$ and $v$ by taking $v_1^\varepsilon-v$ as a test function.
This procedure was enough in the nonstationary case \cite{Miu20_03}.
In the stationary case, however, we still encounter a difficulty due to lack of Gronwall's inequality.
By integration by parts, we see that a bilinear form in the weak form of \eqref{E:SNS_Surf} is
\begin{align*}
  a_g(v_1,v_2) = 2\nu\bigl(D_\Gamma(v_1),D_\Gamma(v_2)\bigr)_{L^2(\Gamma)}, \quad v_1,v_2\in H^1(S^2,TS^2).
\end{align*}
Note that $a_g$ involves $g=g_1-g_0$ in general and here we assume $g\equiv1$.
Also, $H^1(S^2,TS^2)$ is the Sobolev space of tangential vector fields on $S^2$ of class $H^1$ and $D_\Gamma(v_1)$ is the surface strain rate tensor (see Section \ref{SS:Surf} for details).
In order to get a difference estimate for $v_1^\varepsilon$ and $v$ via a suitable weak form, we need to apply the coerciveness of $a_g$ of the form
\begin{align*}
  \|v_1\|_{H^1(S^2)}^2 \leq ca_g(v_1,v_1) = 2\nu c\|D_\Gamma(v_1)\|_{L^2(S^2)}^2, \quad v_1\in H^1(S^2,TS^2).
\end{align*}
This is the Korn inequality on $S^2$ and valid if $v_1$ is orthogonal in $L^2(S^2)$ to
\begin{align} \label{E:Kil_Intro}
  \mathcal{K}(S^2) = \{w(y)=a\times y, \, y\in S^2 \mid a\in\mathbb{R}^3\}.
\end{align}
Note that $\mathcal{K}(S^2)$ is the space of Killing vector fields on $S^2$, and for $w\in\mathcal{K}(S^2)$ we have $D_\Gamma(w)=0$ and $\mathrm{div}_\Gamma w=0$ on $S^2$ (see Section \ref{SS:Surf}).
Hence we need $v_1^\varepsilon-v$ to be orthogonal to $\mathcal{K}(S^2)$ in order to apply the Korn inequality to it.
For the weak solution $v$ to the surface problem \eqref{E:SNS_Surf}, we just assume that it is orthogonal to $\mathcal{K}(S^2)$.
However, we cannot impose the same assumption on $v_1^\varepsilon$, since it is constructed from the solution $u^\varepsilon$ to the bulk problem \eqref{E:Perfect}.
To overcome this difficulty, we use the fact that $\mathcal{K}(S^2)$ is a linear space of dimension three.
Then we can take an orthonormal basis $\{w_1,w_2,w_3\}$ of $\mathcal{K}(S^2)$ with respect to the $L^2(S^2)$-inner product, and by setting
\begin{align*}
  v^\varepsilon = v_1^\varepsilon-\sum_{k=1}^3(v_1^\varepsilon,w_k)_{L^2(S^2)}w_k \quad\text{on}\quad S^2,
\end{align*}
we see that $v^\varepsilon$ is orthogonal to $\mathcal{K}(S^2)$.
Hence we can apply the Korn inequality to $v^\varepsilon-v$ and get a difference estimate for $v^\varepsilon$ and $v$.
It remains to show that $v^\varepsilon$ is sufficiently close to $v_1^\varepsilon$ as $\varepsilon\to0$.
To this end, we use the assumption \eqref{E:OrR_Intro} on $u^\varepsilon$.
By the expressions of $\mathcal{R}$ in \eqref{E:OrR_Intro} and $\mathcal{K}(S^2)$ in \eqref{E:Kil_Intro}, we can take $w=w_k$ in \eqref{E:OrR_Intro} for $k=1,2,3$.
Then, taking the average of $(u^\varepsilon,w_k)_{L^2(\Omega_\varepsilon)}=0$ in the thin direction, we find that $(M_\tau u^\varepsilon,w_k)_{L^2(S^2)}$ is sufficiently small as $\varepsilon\to0$ (see Lemma \ref{L:Ave_RPe}).
Moreover, since
\begin{align*}
  (v_1^\varepsilon,w_k)_{L^2(S^2)} = (M_\tau u^\varepsilon-\nabla_\Gamma q^\varepsilon,w_k)_{L^2(S^2)} = (M_\tau u^\varepsilon,w_k)_{L^2(S^2)}
\end{align*}
by the construction of $v_1^\varepsilon$, integration by parts, and $\mathrm{div}_\Gamma w_k=0$, we can conclude that $v^\varepsilon$ is sufficiently close to $v_1^\varepsilon$ and thus to $M_\tau u^\varepsilon$ as $\varepsilon\to0$, which proves \eqref{E:Er_Intro} when combined with the difference estimate for $v^\varepsilon$ and $v$.
In the case of a general curved thin domain and limit surface, we impose suitable assumptions which enable us to proceed as above (see Assumptions \ref{Asmp_1} and \ref{Asmp_2}).
Also, when $g\not\equiv1$, we use the weighted inner product $(g\cdot,\cdot)_{L^2(S^2)}$ at several points in the above arguments.
For details, see Section \ref{SS:Rep}.

Let us also give a remark on the condition $\|v\|_{H^1(\Gamma)}\leq\rho$.
In the proof of the difference estimate for $v^\varepsilon$ and $v$, we test $v^\varepsilon-v$ in a suitable weak form.
Then the resulting equality involves the term
\begin{align*}
  b_g(v^\varepsilon-v,v,v^\varepsilon-v) = -\bigl(g(v^\varepsilon-v)\otimes v,\nabla_\Gamma(v^\varepsilon-v)\bigr)_{L^2(\Gamma)},
\end{align*}
where $b_g$ is a trilinear form corresponding to the term $g\overline{\nabla}_vv$ in \eqref{E:Lim_Eq}, and we estimate this term by the H\"{o}lder and Ladyzhenskaya inequalities to get (see Lemma \ref{L:TLS_BoAs})
\begin{align*}
  |b_g(v^\varepsilon-v,v,v^\varepsilon-v)| \leq c\|v\|_{H^1(\Gamma)}\|v^\varepsilon-v\|_{H^1(\Gamma)}^2.
\end{align*}
Using this estimate and the Korn and Young inequalities, we a priori obtain
\begin{align*}
  \Bigl(c_1-c_2\|v\|_{H^1(\Gamma)}\Bigr)\|v^\varepsilon-v\|_{H^1(\Gamma)}^2 \leq c_3\Bigl(\delta(\varepsilon)^2+\|f-M_\tau\mathbb{P}_\varepsilon f^\varepsilon\|_{H^{-1}(\Gamma,T\Gamma)}^2\Bigr)
\end{align*}
with some $c_1,c_2,c_3>0$.
Hence we impose $\|v\|_{H^1(\Gamma)}\leq\rho$ to ensure that $c_1-c_2\|v\|_{H^1(\Gamma)}$ is bounded below by some positive constant.
For the actual proof, see Section \ref{SS:Pf_TErr}.

\subsection{Literature overview} \label{SS:In_Lit}
Fluid flows in thin domains appear in many problems of natural sciences like lubrication and geophysical fluid dynamics.
The Navier--Stokes equations in three-dimensional (3D) thin domains have been studied for a long time.
In the case of 3D thin domains around 2D limit sets, there are many works on the nonstationary problem where the authors mainly study the global-in-time existence of a strong solution when the thickness of a thin domain is small and the behavior of a solution as the thickness of a thin domain tends to zero.
The nonstationary problem in a flat thin domain $\Omega_\varepsilon=\omega\times(0,\varepsilon)$ around a 2D bounded domain $\omega$ was studied first by Raugel and Sell \cite{RauSel93_03,RauSel93_01,RauSel94_02} and then by many others (see \cite{TemZia96,MoTeZi97,Ift99,Mon99,IftRau01,KukZia06,Hu07,KukZia07} and the references cited therein).
Also, other kinds of thin domains were considered in the study of the nonstationary problem like a flat thin domain with nonflat top and bottom boundaries \cite{IfRaSe07,Hoa10,HoaSel10,Hoa13}, a thin spherical shell between two concentric spheres \cite{TemZia97}, and a curved thin domain around a general closed surface \cite{Miu20_03,Miu21_02,Miu22_01}.
On the other hand, there are some works that study the asymptotic behavior of a solution to the stationary problem in flat thin domains around 2D domains (see e.g. \cite{Naz90,CaLuSu13,CaLuSu20}), but it seems that the stationary problem has not been studied yet in the case of curved thin domains around surfaces.
We also refer to \cite{NazPil90,Naz00,Mar01} for asymptotic analysis of the stationary problem in thin tubes and networks.

As we mentioned in Section \ref{SS:In_Pro}, the limit equations \eqref{E:Lim_Eq} reduce to the surface Navier--Stokes equations \eqref{E:SNS_Surf} when $g\equiv1$ and $\gamma^0=\gamma^1=0$, and the equations \eqref{E:SNS_Surf} are further equivalent to the Navier--Stokes equations \eqref{E:SNS_Mani} on an abstract Riemannian manifold.
The study of the Navier--Stokes equations on surfaces and manifolds have been done by several authors.
We refer to \cite{Prie94,DinMit04,ChaYon13,KorWen18} for the stationary problem and to \cite{Tay92,Nag99,MitTay01,KheMis12,ChaCzu13,ChaCzu16,Pie17,SamTuo20,PrSiWi21} for the nonstationary problem.
Also, the Navier--Stokes equations on a moving surface have been proposed in several works \cite{KoLiGi17,JaOlRe18,Miu18} by different methods.
Those equations are a coupled system of the motion of a surface and the evolution of the tangential fluid velocity.
When the motion of a surface is given, the wellposedness of the tangential Navier--Stokes equations on a moving surface was established in the recent work \cite{OlReZh22}, but the wellposedness of the full system is still open.

\subsection{Organization of this paper} \label{SS:In_Org}
The rest of this paper is organized as follows.
In Section \ref{S:Main}, we fix notations and state the main results.
Section \ref{S:Tool} provides fundamental tools for analysis of functions on $\Gamma$ and $\Omega_\varepsilon$.
We discuss the existence and uniqueness of a weak solution to the limit equations \eqref{E:Lim_Eq} in Section \ref{S:Sol_Lim}.
In Section \ref{S:Bulk}, we derive some explicit estimates in terms of $\varepsilon$ for a solution to the bulk equations \eqref{E:SNS_CTD}.
Some uniqueness results for a solution to \eqref{E:SNS_CTD} are also given in that section.
In Section \ref{S:Diff_Est}, we give the proofs of the main results.
Section \ref{S:Pf_Sob} is devoted to the proof of Lemma \ref{L:Sob_CTD}.

\section{Notations and main results} \label{S:Main}
We fix notations and state the main results.
Throughout this paper, we fix a coordinate system of $\mathbb{R}^3$ and write $x_i$, $i=1,2,3$ for the $i$-th component of a point $x\in\mathbb{R}^3$ under the fixed coordinate system.
Also, $c$ denotes a general positive constant independent of $\varepsilon$.

\subsection{Notations on vectors and matrices} \label{SS:Vect}
We write a vector and a matrix as
\begin{align*}
  a =
  \begin{pmatrix}
    a_1 \\ a_2 \\ a_3
  \end{pmatrix}
  = (a_1,a_2,a_3)^T\in\mathbb{R}^3, \quad A = (A_{ij})_{i,j} =
  \begin{pmatrix}
    A_{11} & A_{12} & A_{13} \\
    A_{21} & A_{22} & A_{23} \\
    A_{31} & A_{32} & A_{33}
  \end{pmatrix}
  \in\mathbb{R}^{3\times3}.
\end{align*}
Also, we write $A^T$ for the transpose of $A$ and $I_3$ for the $3\times 3$ identity matrix.
We define the tensor product of $a,b\in\mathbb{R}^3$ by $a\otimes b=(a_ib_j)_{i,j}$.
Also, for $A,B\in\mathbb{R}^{3\times3}$ we set the inner product $A:B=\mathrm{tr}[A^TB]$ and the Frobenius norm $|A|=\sqrt{A:A}$.
For 3D vector fields $u$ and $\varphi$ on an open subset of $\mathbb{R}^3$, let
\begin{gather*}
  \nabla u =
  \begin{pmatrix}
    \partial_1u_1 & \partial_1u_2 & \partial_1u_3 \\
    \partial_2u_1 & \partial_2u_2 & \partial_2u_3 \\
    \partial_3u_1 & \partial_3u_2 & \partial_3u_3
  \end{pmatrix}, \quad
  |\nabla^2 u|^2 = \sum_{i,j,k=1}^3|\partial_i\partial_ju_k|^2, \\
  (\varphi\cdot\nabla)u=(\varphi\cdot\nabla u_1,\varphi\cdot\nabla u_2,\varphi\cdot\nabla u_3)^T.
\end{gather*}
Here $\partial_i=\partial/\partial x_i$, $u=(u_1,u_2,u_3)^T$, and $a\cdot b$ is the inner product of $a,b\in\mathbb{R}^3$.

\subsection{Closed surface} \label{SS:Surf}
Let $\Gamma$ be a closed (i.e. compact and without boundary), connected, and oriented surface in $\mathbb{R}^3$ of class $C^l$ with $l\geq2$.
We assume that $\Gamma$ is the boundary of a bounded domain $\Omega$ in $\mathbb{R}^3$ and write $n$ for the unit outward normal vector field of $\Gamma$ which points from $\Omega$ into $\mathbb{R}^3\setminus\Omega$.
Let $d$ be the signed distance function from $\Gamma$ increasing in the direction of $n$.
Also, let $\kappa_1$ and $\kappa_2$ be the principal curvatures of $\Gamma$.
By the regularity of $\Gamma$, we see that $\kappa_1$ and $\kappa_2$ are of class $C^{l-2}$ on $\Gamma$ and thus bounded on $\Gamma$.
Hence there exists a tubular neighborhood $N=\{x\in\mathbb{R}^3 \mid -\delta<d(x)<\delta\}$ of radius $\delta>0$ such that for each $x\in N$ there exists a unique $\pi(x)\in\Gamma$ satisfying
\begin{align} \label{E:Fermi}
  x = \pi(x)+d(x)n(\pi(x)), \quad \nabla d(x) = n(\pi(x))
\end{align}
and that $d$ and $\pi$ are of class $C^l$ and $C^{l-1}$ on $\overline{N}$ (see \cite[Section 14.6]{GilTru01}).
Moreover, taking $\delta>0$ sufficiently small, we may assume that
\begin{align} \label{E:PrCu_Bd}
  c^{-1} \leq 1-r\kappa_i(y) \leq c, \quad y\in\Gamma, \, r\in(-\delta,\delta), \, i=1,2
\end{align}
with some constant $c>0$.

Let $P=I_3-n\otimes n$ be the orthogonal projection onto the tangent plane of $\Gamma$.
We define the tangential gradient of $\eta\in C^1(\Gamma)$ by $\nabla_\Gamma\eta=P\nabla\tilde{\eta}$ on $\Gamma$, where $\tilde{\eta}$ is an extension of $\eta$ to $N$, and denote $\underline{D}_i\eta$ by the $i$-th component of $\nabla_\Gamma\eta$ for $i=1,2,3$.
Note that the value of $\nabla_\Gamma\eta$ is independent of the choice of the extension $\tilde{\eta}$.
In particular, if we set $\bar{\eta}=\eta\circ\pi$, which is the constant extension of $\eta$ in the normal direction of $\Gamma$, then
\begin{align} \label{E:CEGr_Su}
  \nabla\bar{\eta}(y) = \nabla_\Gamma\eta(y), \quad \partial_i\bar{\eta}(y) = \underline{D}_i\eta(y), \quad y\in\Gamma, \, i=1,2,3,
\end{align}
since $\nabla\pi(y)=P(y)$ for $y\in\Gamma$ by \eqref{E:Fermi} and $d(y)=0$.
In what follows, a function with an overline always stands for the constant extension of a function on $\Gamma$.

For a (not necessarily tangential) vector field $v=(v_1,v_2,v_3)^T\in C^1(\Gamma)^3$, we define the tangential gradient matrix and the surface divergence of $v$ by
\begin{align*}
  \nabla_\Gamma v =
  \begin{pmatrix}
    \underline{D}_1v_1 & \underline{D}_1v_2 & \underline{D}_1v_3 \\
    \underline{D}_2v_1 & \underline{D}_2v_2 & \underline{D}_2v_3 \\
    \underline{D}_3v_1 & \underline{D}_3v_2 & \underline{D}_3v_3
  \end{pmatrix}, \quad
  \mathrm{div}_\Gamma v = \mathrm{tr}[\nabla_\Gamma v] = \sum_{i=1}^3\underline{D}_iv_i.
\end{align*}
Note that the order of the indices of $\nabla_\Gamma v$ is reversed in some literature.
We also set
\begin{align} \label{E:Surf_SRT}
  D_\Gamma(v) = P(\nabla_\Gamma v)_SP \quad\text{on}\quad \Gamma, \quad (\nabla_\Gamma v)_S = \frac{\nabla_\Gamma v+(\nabla_\Gamma v)^T}{2}
\end{align}
and call $D_\Gamma(v)$ the surface strain rate tensor.
For $v\in C^1(\Gamma)^3$ and $\eta\in C(\Gamma)^3$, let
\begin{align*}
  (\eta\cdot\nabla_\Gamma)v = (\eta\cdot\nabla_\Gamma v_1,\eta\cdot\nabla_\Gamma v_2,\eta\cdot\nabla_\Gamma v_3)^T \quad\text{on}\quad \Gamma.
\end{align*}
When $v$ and $\eta$ are tangential, we write $\overline{\nabla}_\eta v=P(\eta\cdot\nabla_\Gamma)v$ for the covariant derivative of $v$ along $\eta$.
For a $3\times 3$ matrix-valued function $A=(A_{ij})_{i,j}$ on $\Gamma$, we define $\mathrm{div}_\Gamma A$ as a vector field on $\Gamma$ with $j$-th component $[\mathrm{div}_\Gamma A]_j=\sum_{i=1}^3\underline{D}_iA_{ij}$ for $j=1,2,3$.
Let $W=-\nabla_\Gamma n$ be the Weingarten map of $\Gamma$.
By \eqref{E:Fermi}, \eqref{E:CEGr_Su}, and $|n|=1$, we have
\begin{align} \label{E:Wein}
  W = -\nabla^2d = W^T, \quad Wn = -\frac{1}{2}\nabla_\Gamma(|n|^2) = 0 \quad\text{on}\quad \Gamma.
\end{align}
We also define the mean curvature of $\Gamma$ by $H=-\mathrm{div}_\Gamma n$.

Let $\eta,\xi\in C^1(\Gamma)$ and $i=1,2,3$.
Then the integration by parts formula
\begin{align} \label{E:TD_IbP}
  \int_\Gamma(\eta\underline{D}_i\xi+\xi\underline{D}_i\eta)\,d\mathcal{H}^2 = -\int_\Gamma\eta\xi Hn_i\,d\mathcal{H}^2
\end{align}
holds (see e.g. \cite[Lemma 16.1]{GilTru01} and \cite[Theorem 2.10]{DziEll13}), where $\mathcal{H}^2$ is the 2D Hausdorff measure.
Based on this formula, we say that a function $\eta\in L^2(\Gamma)$ has the $i$-th weak tangential derivative if there exists $\eta_i\in L^2(\Gamma)$ such that
\begin{align*}
  \int_\Gamma\eta_i\xi\,d\mathcal{H}^2 = -\int_\Gamma\eta(\underline{D}_i\xi+\xi Hn_i)\,d\mathcal{H}^2
\end{align*}
for all $\xi\in C^1(\Gamma)$.
In this case, we write $\underline{D}_i\eta=\eta_i$ and define the Sobolev spaces
\begin{align*}
  H^1(\Gamma) &= \{\eta\in L^2(\Gamma) \mid \underline{D}_i\eta\in L^2(\Gamma), \, i=1,2,3\}, \\
  H^2(\Gamma) &= \{\eta\in H^1(\Gamma) \mid \underline{D}_i\underline{D}_j\eta\in L^2(\Gamma), \, i,j=1,2,3\}.
\end{align*}
We also define $\|\eta\|_{H^1(\Gamma)}$ and $\|\eta\|_{H^2(\Gamma)}$ as in the flat domain case.

For a function space $\mathcal{X}(\Gamma)$ on $\Gamma$, we denote by
\begin{align*}
  \mathcal{X}(\Gamma,T\Gamma)=\{v\in \mathcal{X}(\Gamma)^3 \mid \text{$v\cdot n=0$ on $\Gamma$}\}
\end{align*}
the space of tangential vector fields on $\Gamma$ of class $\mathcal{X}$.
If $\mathcal{X}=L^2,H^1$, then $\mathcal{X}(\Gamma,T\Gamma)$ is a closed subspace of $\mathcal{X}(\Gamma)^3$.
Let $H^{-1}(\Gamma,T\Gamma)$ be the dual space of $H^1(\Gamma,T\Gamma)$ with duality product $[\cdot,\cdot]_{T\Gamma}$.
We consider $L^2(\Gamma,T\Gamma)$ as a subspace of $H^{-1}(\Gamma,T\Gamma)$ by setting
\begin{align} \label{E:HinS_L2}
  [v,w]_{T\Gamma} = (v,w)_{L^2(\Gamma)}, \quad v\in L^2(\Gamma,T\Gamma), \, w\in H^1(\Gamma,T\Gamma).
\end{align}
For $\eta\in C^1(\Gamma)$ and $f\in H^{-1}(\Gamma,T\Gamma)$, we define $\eta f\in H^{-1}(\Gamma,T\Gamma)$ by
\begin{align} \label{E:HinS_Mul}
  [\eta f,v]_{T\Gamma} = [f,\eta v]_{T\Gamma}, \quad v\in H^1(\Gamma,T\Gamma).
\end{align}
Let $g\in C^1(\Gamma)$.
We define the spaces of weighted solenoidal vector fields on $\Gamma$ by
\begin{align*}
  \mathcal{X}_{g\sigma}(\Gamma,T\Gamma) = \{v\in\mathcal{X}(\Gamma,T\Gamma) \mid \text{$\mathrm{div}_\Gamma(gv)=0$ on $\Gamma$}\}, \quad \mathcal{X} = L^2,H^1.
\end{align*}
Here we consider $\mathrm{div}_\Gamma(gv)$ in the dual space of $H^1(\Gamma)$ when $\mathcal{X}=L^2$.
Also, $\mathcal{X}_{g\sigma}(\Gamma,T\Gamma)$ is a closed subspace of $\mathcal{X}(\Gamma,T\Gamma)$.
For details, see \cite[Sections 3.1 and 6.3]{Miu20_03}.
We also set
\begin{align} \label{E:Def_Kilg}
  \begin{aligned}
    \mathcal{K}(\Gamma) &= \{v\in H^1(\Gamma,T\Gamma) \mid \text{$D_\Gamma(v)=0$ on $\Gamma$}\}, \\
    \mathcal{K}_g(\Gamma) &= \{v\in\mathcal{K}(\Gamma) \mid \text{$v\cdot\nabla_\Gamma g=0$ on $\Gamma$}\}.
  \end{aligned}
\end{align}
It is known that a vector field in $\mathcal{K}(\Gamma)$ preserves the Riemannian metric of $\Gamma$, and such a vector field is called a Killing vector field (see e.g. \cite{Pet16} for details).

\begin{lemma} \label{L:Kil_Sol}
  If $v\in\mathcal{K}(\Gamma)$, then $\mathrm{div}_\Gamma v=0$ on $\Gamma$.
\end{lemma}

\begin{proof}
  Since $v\cdot n=0$ and $(\nabla_\Gamma v)(n\otimes n)=\{(\nabla_\Gamma v)n\}\otimes n$ on $\Gamma$, we have
  \begin{align*}
    (\nabla_\Gamma v)n = \nabla_\Gamma(v\cdot n)-(\nabla_\Gamma n)v = Wv, \quad (\nabla_\Gamma v)(n\otimes n) = (Wv)\otimes n \quad\text{on}\quad \Gamma.
  \end{align*}
  Thus $\nabla_\Gamma v=(\nabla_\Gamma v)P+(Wv)\otimes n$ on $\Gamma$ by $P=I_3-n\otimes n$.
  Moreover,
  \begin{align*}
    \mathrm{tr}[(\nabla_\Gamma v)P] = \mathrm{tr}[P(\nabla_\Gamma v)P] = \mathrm{tr}[D_\Gamma(v)] = 0, \quad \mathrm{tr}[(Wv)\otimes n] = (Wv)\cdot n = 0
  \end{align*}
  on $\Gamma$ by $P\nabla_\Gamma v=\nabla_\Gamma v$, $v\in\mathcal{K}(\Gamma)$, and \eqref{E:Wein}.
  Hence $\mathrm{div}_\Gamma v=\mathrm{tr}[\nabla_\Gamma v]=0$ on $\Gamma$.
\end{proof}

\begin{lemma} \label{L:Kilg_Sub}
  The space $\mathcal{K}_g(\Gamma)$ is a finite dimensional subspace of $H_{g\sigma}^1(\Gamma,T\Gamma)$.
\end{lemma}

\begin{proof}
  Let $v\in\mathcal{K}_g(\Gamma)$.
  Then, by Lemma \ref{L:Kil_Sol} and $\nabla_\Gamma g\cdot v=0$ on $\Gamma$, we have
  \begin{align*}
    \mathrm{div}_\Gamma(gv) = g\,\mathrm{div}_\Gamma v+\nabla_\Gamma g\cdot v = 0 \quad\text{on}\quad \Gamma, \quad\text{i.e.}\quad v\in H_{g\sigma}^1(\Gamma,T\Gamma).
  \end{align*}
  Thus $\mathcal{K}_g(\Gamma)\subset H_{g\sigma}^1(\Gamma,T\Gamma)$.
  It is known (see e.g. \cite[Theorem 8.1.6]{Pet16}) that $\mathcal{K}(\Gamma)$ is a Lie algebra of dimension at most three.
  Hence $\mathcal{K}_g(\Gamma)$ is also of finite dimension.
\end{proof}

Let $\gamma^0,\gamma^1\geq0$ be the friction coefficients appearing in \eqref{E:Lim_Eq}.
We set
\begin{align*}
  \mathcal{X}^{\perp g} = \{v\in L^2(\Gamma,T\Gamma) \mid \text{$(gv,w)_{L^2(\Gamma)}=0$ for all $w\in\mathcal{X}$}\}
\end{align*}
for a subset $\mathcal{X}$ of $L^2(\Gamma,T\Gamma)$ and define
\begin{align*}
  \mathbb{H}_g &=
  \begin{cases}
    L_{g\sigma}^2(\Gamma,T\Gamma) &\text{if $\gamma^0>0$ or $\gamma^1>0$ or $\mathcal{K}_g(\Gamma)=\{0\}$}, \\
    L_{g\sigma}^2(\Gamma,T\Gamma)\cap \mathcal{K}_g(\Gamma)^{\perp g} &\text{if $\gamma^0=\gamma^1=0$ and $\mathcal{K}_g(\Gamma)\neq\{0\}$},
  \end{cases} \\
  \mathbb{V}_g &= \mathbb{H}_g\cap H^1(\Gamma,T\Gamma),
\end{align*}
which are closed subspaces of $L_{g\sigma}^2(\Gamma,T\Gamma)$ and $H_{g\sigma}^1(\Gamma,T\Gamma)$, respectively.
It turns out (see Lemma \ref{L:BLS_BoCo}) that a biliear form appearing in a weak form of \eqref{E:Lim_Eq} is coercive on $\mathbb{V}_g$, which is crucial for the proof of Theorem \ref{T:Error} given below.

\subsection{Curved thin domain} \label{SS:CTD}
Now let $\Gamma$ be of class $C^5$ and $g_0,g_1\in C^4(\Gamma)$.
In what follows, we set $g=g_1-g_0$ and assume that there exists a constant $c>0$ such that
\begin{align} \label{E:G_Inf}
  g \geq c \quad\text{on}\quad \Gamma.
\end{align}
For a sufficiently small $\varepsilon>0$, we define a curved thin domain $\Omega_\varepsilon$ in $\mathbb{R}^3$ and its inner and outer boundaries $\Gamma_\varepsilon^0$ and $\Gamma_\varepsilon^1$ by \eqref{E:Def_CTD} and write $\Gamma_\varepsilon=\Gamma_\varepsilon^0\cup\Gamma_\varepsilon^1$ for the whole boundary of $\Omega_\varepsilon$.
Note that $\Gamma_\varepsilon$ is of class $C^4$ by the regularity of $\Gamma$, $g_0$, and $g_1$.
This fact is required for the proof of the uniform norm equivalence of the Stokes operator on $\Omega_\varepsilon$ given in \cite{Miu22_01} (see Lemma \ref{L:Ae_Equi} for the statement).
In this paper, we employ the result of \cite{Miu22_01}, but do not use the regularity of $\Gamma_\varepsilon$ explicitly.

Since $g_0$ and $g_1$ are bounded on $\Gamma$, we can take a sufficiently small $\tilde{\varepsilon}\in(0,1)$ so that $\varepsilon|g_i|<\delta$ on $\Gamma$ for all $\varepsilon\in(0,\tilde{\varepsilon})$ and $i=0,1$, where $\delta$ is the radius of the tubular neighborhood $N$ of $\Gamma$ given in Section \ref{SS:Surf}.
Hence $\overline{\Omega_\varepsilon}\subset N$ and we can use the results of Section \ref{SS:Surf} in $\overline{\Omega_\varepsilon}$ for all $\varepsilon\in(0,\tilde{\varepsilon})$.
From now on, we always assume $\varepsilon\in(0,\tilde{\varepsilon})$.

We define the solenoidal spaces on $\Omega_\varepsilon$ by
\begin{align} \label{E:Sol_CTD}
  \begin{aligned}
    L_\sigma^2(\Omega_\varepsilon) &= \{u\in L^2(\Omega_\varepsilon)^3\mid \text{$\mathrm{div}\,u=0$ in $\Omega_\varepsilon$, $u\cdot n_\varepsilon=0$ on $\Gamma_\varepsilon$}\}, \\
    H_{n,\sigma}^1(\Omega_\varepsilon) &= L_\sigma^2(\Omega_\varepsilon)\cap H^1(\Omega_\varepsilon)^3.
  \end{aligned}
\end{align}
Note that $H_{n,\sigma}^1(\Omega_\varepsilon)$ is different from $H_{0,\sigma}^1(\Omega_\varepsilon)$ used in the literature, which is the space of divergence-free vector fields in $H^1(\Omega_\varepsilon)^3$ totally vanishing on the boundary.

\subsection{Main results} \label{SS:Main}
To state the main results, we further require some assumptions and function spaces.
We define spaces of infinitesimal rigid displacements by
\begin{align} \label{E:Def_Rota}
  \begin{aligned}
    \mathcal{R} &= \{w(x)=a\times x+b, \, x\in\mathbb{R}^3 \mid a,b\in\mathbb{R}^3, \, \text{$w|_\Gamma\cdot n=0$ on $\Gamma$}\}, \\
    \mathcal{R}_i &= \{w\in\mathcal{R} \mid \text{$w|_\Gamma\cdot\nabla_\Gamma g_i=0$ on $\Gamma$}\}, \quad i=0,1, \\
    \mathcal{R}_g &= \{w\in\mathcal{R} \mid \text{$w|_\Gamma\cdot\nabla_\Gamma g=0$ on $\Gamma$}\} \quad (g=g_1-g_0).
  \end{aligned}
\end{align}
Note that $\mathcal{R}$ and thus $\mathcal{R}_0$, $\mathcal{R}_1$, and $\mathcal{R}_g$ are of finite dimension.
We write
\begin{align*}
  \mathcal{R}|_\Gamma = \{w|_\Gamma \mid w\in\mathcal{R}\}, \quad \mathcal{R}_g|_\Gamma = \{w|_\Gamma \mid w\in\mathcal{R}_g\}.
\end{align*}
Then we easily observe that $\mathcal{R}|_\Gamma\subset\mathcal{K}(\Gamma)$, where $\mathcal{K}(\Gamma)$ is given in \eqref{E:Def_Kilg}.
The spaces $\mathcal{K}(\Gamma)$ and $\mathcal{R}$ represent the intrinsic and extrinsic infinitesimal symmetry of $\Gamma$, respectively.
It is known that $\mathcal{R}|_\Gamma=\mathcal{K}(\Gamma)$ when
\begin{itemize}
  \item $\Gamma$ is axially symmetric (see e.g. \cite[Lemma E.3]{Miu22_01}),
  \item $\Gamma$ is a closed and convex (by the Cohn-Vossen rigidity theorem, see \cite{Spi79_05}),
  \item the genus of $\Gamma$ is greater than one, in which case $\mathcal{K}(\Gamma)=\{0\}$ (see e.g. \cite{Shi18}).
\end{itemize}
Let $\gamma_\varepsilon^0,\gamma_\varepsilon^1\geq0$ be the friction coefficients appearing in \eqref{E:Fric}.
We make assumptions on $\gamma_\varepsilon^0$, $\gamma_\varepsilon^1$, and the function spaces given in \eqref{E:Def_Kilg} and \eqref{E:Def_Rota} as follows.

\begin{assumption} \label{Asmp_1}
  There exists a constant $c>0$ such that
  \begin{align} \label{E:Fr_Up}
    \gamma_\varepsilon^0 \leq c\varepsilon, \quad \gamma_\varepsilon^1 \leq c\varepsilon
  \end{align}
  for all $\varepsilon\in(0,1)$.
\end{assumption}

\begin{assumption} \label{Asmp_2}
  Either of the following conditions is satisfied:
  \begin{itemize}
    \item[(A1)] there exists a constant $c>0$ such that
    \begin{align*}
      \gamma_\varepsilon^0 \geq c\varepsilon \quad\text{for all}\quad \varepsilon\in(0,1) \quad\text{or}\quad \gamma_\varepsilon^1 \geq c\varepsilon \quad\text{for all}\quad \varepsilon\in(0,1),
    \end{align*}
    \item[(A2)] $\mathcal{K}_g(\Gamma)=\{0\}$,
    \item[(A3)] $\mathcal{R}_g=\mathcal{R}_0\cap\mathcal{R}_1$, $\mathcal{R}_g|_\Gamma=\mathcal{K}_g(\Gamma)$, and $\gamma_\varepsilon^0=\gamma_\varepsilon^1=0$ for all $\varepsilon\in(0,1)$.
  \end{itemize}
\end{assumption}

We impose these assumptions in this section and Sections \ref{S:Bulk} and \ref{S:Diff_Est}.
They are required to make a bilinear form appearing in a weak form of \eqref{E:SNS_CTD} bounded and coercive uniformly in $\varepsilon$ on the function space $\mathcal{V}_\varepsilon$ given below.
See Lemma \ref{L:BL_Unif} for details.

Under Assumptions \ref{Asmp_1} and \ref{Asmp_2}, we define
\begin{align} \label{E:Def_Heps}
  \begin{aligned}
    \mathcal{H}_\varepsilon &=
    \begin{cases}
      L_\sigma^2(\Omega_\varepsilon) &\text{if the condition (A1) or (A2) is imposed}, \\
      L_\sigma^2(\Omega_\varepsilon)\cap\mathcal{R}_g^\perp &\text{if the condition (A3) is imposed},
  \end{cases} \\
  \mathcal{V}_\varepsilon &= \mathcal{H}_\varepsilon\cap H^1(\Omega_\varepsilon)^3.
    \end{aligned}
\end{align}
Here $\mathcal{R}_g^\perp$ is the orthogonal complement of $\mathcal{R}_g$ in $L^2(\Omega_\varepsilon)^3$.
Also, $\mathcal{H}_\varepsilon$ and $\mathcal{V}_\varepsilon$ are closed subspaces of $L^2(\Omega_\varepsilon)^3$ and $H^1(\Omega_\varepsilon)^3$, respectively.
We denote $\mathbb{P}_\varepsilon$ by the orthogonal projection from $L^2(\Omega_\varepsilon)^3$ onto $\mathcal{H}_\varepsilon$.

Now we are ready to state the main results.

\begin{theorem} \label{T:Error}
  Under Assumptions \ref{Asmp_1} and \ref{Asmp_2}, let $f^\varepsilon\in L^2(\Omega_\varepsilon)^3$ satisfy
  \begin{align} \label{E:Fth_Comp}
    f^\varepsilon \in \mathcal{R}_g^\perp \quad \text{if the condition (A3) is imposed}
  \end{align}
  and let $u^\varepsilon$ be a weak solution to \eqref{E:SNS_CTD}.
  Also, let $f\in H^{-1}(\Gamma,T\Gamma)$ satisfy
  \begin{align} \label{E:Fsu_Comp}
    [gf,w]_{T\Gamma} = 0 \quad\text{for all}\quad w\in\mathcal{K}_g(\Gamma) \quad\text{if $\gamma^0=\gamma^1=0$ and $\mathcal{K}_g(\Gamma)\neq\{0\}$}
  \end{align}
  and let $v$ be a weak solution to \eqref{E:Lim_Eq}.
  Suppose that
  \begin{itemize}
    \item[(a)] there exist constants $c_1,c_2>0$ and $\alpha\in(0,1]$ such that
    \begin{align*}
      \|\mathbb{P}_\varepsilon f^\varepsilon\|_{L^2(\Omega_\varepsilon)}^2 \leq c_1\varepsilon^{-1+\alpha}, \quad \|M_\tau\mathbb{P}_\varepsilon f^\varepsilon\|_{H^{-1}(\Gamma,T\Gamma)}^2 \leq c_2
    \end{align*}
    for all $\varepsilon\in(0,1)$ sufficiently small,
    \item[(b)] $u^\varepsilon\in\mathcal{V}_\varepsilon$ for all $\varepsilon\in(0,1)$ sufficiently small and $v\in\mathbb{V}_g$, and
    \item[(c)] $\gamma^0>0$ or $\gamma^1>0$ if the condition (A1) of Assumption \ref{Asmp_2} is imposed.
  \end{itemize}
  Then there exist constants $c,\rho>0$ independent of $\varepsilon$, $f^\varepsilon$, $u^\varepsilon$, $f$, and $v$ such that
  \begin{align} \label{E:Error}
    \|M_\tau u^\varepsilon-v\|_{H^1(\Gamma)} \leq c\Bigl(\delta(\varepsilon)+\|M_\tau\mathbb{P}_\varepsilon f^\varepsilon-f\|_{H^{-1}(\Gamma,T\Gamma)}\Bigr)
  \end{align}
  for all $\varepsilon\in(0,1)$ sufficiently small provided that $\|v\|_{H^1(\Gamma)}\leq\rho$, where
  \begin{align*}
    \delta(\varepsilon) = \varepsilon^{\alpha/4}+\sum_{i=0,1}\left|\frac{\gamma_\varepsilon^i}{\varepsilon}-\gamma^i\right|.
  \end{align*}
  In particular, if in addition $\gamma_\varepsilon^i/\varepsilon\to\gamma^i$ for $i=0,1$ and $M_\tau\mathbb{P}_\varepsilon f^\varepsilon\to f$ strongly in $H^{-1}(\Gamma,T\Gamma)$ as $\varepsilon\to0$, then $M_\tau u^\varepsilon\to v$ strongly in $H^1(\Gamma,T\Gamma)$ as $\varepsilon\to0$.
\end{theorem}

The proof of Theorem \ref{T:Error} is given in Section \ref{SS:Pf_TErr}.
Note that \eqref{E:Fth_Comp} and \eqref{E:Fsu_Comp} are the compatibility conditions for the existence of weak solutions to \eqref{E:SNS_CTD} and \eqref{E:Lim_Eq}, respectively (see Propositions \ref{P:LW_Comp} and \ref{P:F_Comp}).
Also, the constant $\rho$ is explicitly given by other constants in the proof of Theorem \ref{T:Error}.
It turns out that the condition $\|v\|_{H^1(\Gamma)}\leq\rho$ implies the uniqueness of a weak solution to \eqref{E:Lim_Eq} in the class $\mathbb{V}_g$ (see Remark \ref{R:Pf_TErr}).

We also have a difference estimate in $\Omega_\varepsilon$.
Recall that $\bar{\eta}$ denotes the constant extension of a function $\eta$ on $\Gamma$ in the normal direction of $\Gamma$.
Also, for a function $\varphi$ on $\Omega_\varepsilon$, we write $\partial_n\varphi=(\bar{n}\cdot\nabla)\varphi$ for the derivative of $\varphi$ in the normal direction of $\Gamma$.

\begin{theorem} \label{T:Er_CE}
  Under the assumptions of Theorem \ref{T:Error}, we have
  \begin{multline} \label{E:Er_CE}
    \varepsilon^{-1/2}\left(\|u^\varepsilon-\bar{v}\|_{L^2(\Omega_\varepsilon)}+\Bigl\|\overline{P}\nabla u^\varepsilon-\overline{\nabla_\Gamma v}\Bigr\|_{L^2(\Omega_\varepsilon)}+\Bigl\|\partial_nu^\varepsilon-\overline{V}\Bigr\|_{L^2(\Omega_\varepsilon)}\right) \\
    \leq c\Bigl(\delta(\varepsilon)+\|M_\tau\mathbb{P}_\varepsilon f^\varepsilon-f\|_{H^{-1}(\Gamma,T\Gamma)}\Bigr)
  \end{multline}
  for all $\varepsilon\in(0,1)$ sufficiently small, where
  \begin{align*}
    V = -Wv+\frac{1}{g}(v\cdot\nabla_\Gamma g)n \quad\text{on}\quad \Gamma
  \end{align*}
  and $c>0$ is a constant independent of $\varepsilon$, $f^\varepsilon$, $u^\varepsilon$, $f$, and $v$.
\end{theorem}

The proof of Theorem \ref{T:Er_CE} is given in Section \ref{SS:Pf_ErCE}.
Note that the left-hand side of \eqref{E:Er_CE} is divided by $\varepsilon^{1/2}$ since the $L^2(\Omega_\varepsilon)$-norm involves the square root of the thickness of $\Omega_\varepsilon$.
Moreover, since
\begin{align*}
  \nabla u^\varepsilon = \overline{P}\nabla u^\varepsilon+(\bar{n}\otimes\bar{n})\nabla u^\varepsilon = \overline{P}\nabla u^\varepsilon+\bar{n}\otimes\partial_n u^\varepsilon \quad\text{in}\quad \Omega_\varepsilon,
\end{align*}
the estimate \eqref{E:Er_CE} gives the approximation of the whole $\nabla u^\varepsilon$ by functions on $\Gamma$.
We also note that the surface vector field $V$ in \eqref{E:Er_CE} does not vanish in general.
Hence $\partial_nu^\varepsilon$, the derivative of $u^\varepsilon$ in the normal direction of $\Gamma$, is not small as $\varepsilon\to0$ in general, even though $u^\varepsilon$ is close to the constant extension of $v$ in the normal direction of $\Gamma$.

\section{Fundamental tools} \label{S:Tool}
We give fundamental tools for analysis of functions on $\Gamma$ and $\Omega_\varepsilon$.
Recall that we denote by $c$ a general positive constant independent of $\varepsilon$ and by $\bar{\eta}$ the constant extension of a function $\eta$ on $\Gamma$ in the normal direction of $\Gamma$.

\subsection{Inequalities on the surface} \label{SS:In_Surf}
We give the Ladyzhenskaya and Korn inequalities on $\Gamma$.
Let $\mathcal{K}_g(\Gamma)$ be the function space given by \eqref{E:Def_Kilg}.

\begin{lemma} \label{L:Lad_Ineq}
  There exists a constant $c>0$ such that
  \begin{align} \label{E:Lad_Ineq}
    \|\eta\|_{L^4(\Gamma)} \leq c\|\eta\|_{L^2(\Gamma)}^{1/2}\|\eta\|_{H^1(\Gamma)}^{1/2}
  \end{align}
  for all $\eta\in H^1(\Gamma)$.
\end{lemma}

\begin{proof}
  We refer to \cite[Lemma 4.1]{Miu21_02} for the proof.
\end{proof}

\begin{lemma} \label{L:Korn_Surf}
  There exists a constant $c>0$ such that
  \begin{align} \label{E:Korn_Surf}
    \|v\|_{H^1(\Gamma)} \leq c\Bigl(\|D_\Gamma(v)\|_{L^2(\Gamma)}+\|v\|_{L^2(\Gamma)}\Bigr)
  \end{align}
  for all $v\in H^1(\Gamma,T\Gamma)$.
\end{lemma}

\begin{proof}
  The inequality \eqref{E:Korn_Surf} is shown in the proof of \cite[Lemma 4.1]{JaOlRe18} (see (4.8) in \cite{JaOlRe18}) under the $C^2$-regularity of $\Gamma$, which uses local coordinates of $\Gamma$.
  Note that the order of the indices of the (tangential) gradient matrices is reversed in \cite{JaOlRe18} and that $\nabla_\Gamma v$, $(\nabla_\Gamma v)_S$, and $D_\Gamma(v)$ in this paper are denoted by $\nabla_P\mathbf{u}$, $\mathbf{e}_s(\mathbf{u})$, and $E_s(\mathbf{u})$ in \cite{JaOlRe18}.
  We also refer to \cite[Lemma 4.2]{Miu20_03} for the proof of \eqref{E:Korn_Surf} without using local coordinates of $\Gamma$, but it requires the $C^3$-regularity of $\Gamma$ since it uses a density argument for which the $C^2$-regularity of $P$ is necessary.
\end{proof}

\begin{lemma} \label{L:Korn_G}
  There exists a constant $c>0$ such that
  \begin{align} \label{E:Korn_G}
    \|v\|_{H^1(\Gamma)} \leq c\Bigl(\|D_\Gamma(v)\|_{L^2(\Gamma)}+\|v\cdot\nabla_\Gamma g\|_{L^2(\Gamma)}\Bigr)
  \end{align}
  for all $v\in H^1(\Gamma,T\Gamma)$ satisfying
  \begin{align} \label{E:KoG_Per}
    (gv,w)_{L^2(\Gamma)} = 0 \quad\text{for all}\quad w\in\mathcal{K}_g(\Gamma).
  \end{align}
\end{lemma}

\begin{proof}
  By \eqref{E:Korn_Surf}, it is sufficient for \eqref{E:Korn_G} to show that
  \begin{align} \label{Pf_KG:Goal}
    \|v\|_{L^2(\Gamma)} \leq c\Bigl(\|D_\Gamma(v)\|_{L^2(\Gamma)}+\|v\cdot\nabla_\Gamma g\|_{L^2(\Gamma)}\Bigr)
  \end{align}
  for all $v\in H^1(\Gamma,T\Gamma)$ satisfying \eqref{E:KoG_Per}.
  Assume to the contrary that there exist vector fields $v_k\in H^1(\Gamma,T\Gamma)$, $k\in\mathbb{N}$ satisfying \eqref{E:KoG_Per} and
  \begin{align} \label{Pf_KG:Cont}
    \|v_k\|_{L^2(\Gamma)} > k\Bigl(\|D_\Gamma(v_k)\|_{L^2(\Gamma)}+\|v_k\cdot\nabla_\Gamma g\|_{L^2(\Gamma)}\Bigr).
  \end{align}
  Since $\|v_k\|_{L^2(\Gamma)}\neq0$, we may assume $\|v_k\|_{L^2(\Gamma)}=1$ by replacing $v_k$ by $v_k/\|v_k\|_{L^2(\Gamma)}$.
  Then it follows from \eqref{Pf_KG:Cont} that
  \begin{align} \label{Pf_KG:Co_Con}
    \lim_{k\to\infty}\|D_\Gamma(v_k)\|_{L^2(\Gamma)} = \lim_{k\to\infty}\|v_k\cdot\nabla_\Gamma g\|_{L^2(\Gamma)} = 0.
  \end{align}
  Moreover, by \eqref{E:Korn_Surf}, \eqref{Pf_KG:Cont}, and $\|v_k\|_{L^2(\Gamma)}=1$, we see that $\{v_k\}_{k=1}^\infty$ is bounded in $H^1(\Gamma,T\Gamma)$.
  By this fact and the compact embedding from $H^1$ into $L^2$, the sequence $\{v_k\}_{k=1}^\infty$ converges (up to a subsequence) to some $v$ weakly in $H^1(\Gamma,T\Gamma)$ and strongly in $L^2(\Gamma,T\Gamma)$.
  Thus
  \begin{align*}
    \|v\|_{L^2(\Gamma)} = \lim_{k\to\infty}\|v_k\|_{L^2(\Gamma)} = 1.
  \end{align*}
  However, since $v\in\mathcal{K}_g(\Gamma)$ by \eqref{Pf_KG:Co_Con} and the weak convergence of $\{v_k\}_{k=1}^\infty$ to $v$ in $H^1(\Gamma,T\Gamma)$, and since $v_k$ satisfies \eqref{E:KoG_Per} for all $k\in\mathbb{N}$, we have
  \begin{align*}
    \|g^{1/2}v\|_{L^2(\Gamma)}^2 = (gv,v)_{L^2(\Gamma)} = \lim_{k\to\infty}(gv_k,v)_{L^2(\Gamma)} = 0.
  \end{align*}
  Thus $v=0$ on $\Gamma$ by \eqref{E:G_Inf}, which contradicts $\|v\|_{L^2(\Gamma)}=1$.
  Hence \eqref{Pf_KG:Goal} is valid.
\end{proof}

\begin{remark} \label{R:Korn_G}
  The condition \eqref{E:KoG_Per} can be seen as an orthogonality condition with respect to the weighted inner product $(g\,\cdot,\cdot)_{L^2(\Gamma)}$.
  We can also show that \eqref{E:Korn_G} holds provided that $(v,w)_{L^2(\Gamma)}=0$ for all $w\in\mathcal{K}_g(\Gamma)$ instead of \eqref{E:KoG_Per}, but the condition \eqref{E:KoG_Per} is appropriate for our purpose (see Section \ref{S:Diff_Est}).
  When $g$ is constant and thus $\nabla_\Gamma g=0$ on $\Gamma$, Lemma \ref{L:Korn_G} gives a standard Korn inequality on the surface $\Gamma$ (see \cite[Lemma 4.1]{JaOlRe18}).
\end{remark}

\subsection{Inequalities on the thin domain} \label{SS:In_CTD}
For $y\in\Gamma$ and $r\in(-\delta,\delta)$, let
\begin{align} \label{E:Def_Ja}
  J(y,r) = \det[I_3-rW(y)] = \{1-r\kappa_1(y)\}\{1-r\kappa_2(y)\}.
\end{align}
This $J$ is the Jacobian in the change of variables formula (see e.g. \cite[Section 14.6]{GilTru01})
\begin{align} \label{E:CoV_CTD}
  \int_{\Omega_\varepsilon}\varphi(x)\,dx = \int_{\Gamma}\int_{\varepsilon g_0(y)}^{\varepsilon g_1(y)}\varphi(y+rn(y))J(y,r)\,dr\,d\mathcal{H}^2(y)
\end{align}
for a function $\varphi$ on $\Omega_\varepsilon$.
Moreover, we have
\begin{align} \label{E:Jacob}
  c^{-1} \leq J(y,r) \leq c, \quad |J(y,r)-1| \leq c|r| \leq c\varepsilon
\end{align}
for all $y\in\Gamma$ and $r\in[\varepsilon g_0(y),\varepsilon g_1(y)]$, since \eqref{E:PrCu_Bd} holds and $g_0$, $g_1$, $\kappa_1$, and $\kappa_2$ are bounded on $\Gamma$.
Using these facts and \eqref{E:G_Inf}, we easily get the following inequalities.

\begin{lemma} \label{L:CoV_Thin}
  For all $\varphi\in L^2(\Omega_\varepsilon)$ we have
  \begin{align} \label{E:CoV_L2}
    c^{-1}\|\varphi\|_{L^2(\Omega_\varepsilon)}^2 \leq \int_{\Gamma}\int_{\varepsilon g_0(y)}^{\varepsilon g_1(y)}|\varphi(y+rn(y))|^2\,dr\,d\mathcal{H}^2(y) \leq c\|\varphi\|_{L^2(\Omega_\varepsilon)}^2.
  \end{align}
  Also, for all $\eta\in L^2(\Gamma)$ and its constant extension $\bar{\eta}$ in the normal direction of $\Gamma$,
  \begin{align} \label{E:CE_L2CTD}
    c^{-1}\varepsilon^{1/2}\|\eta\|_{L^2(\Gamma)} \leq \|\bar{\eta}\|_{L^2(\Omega_\varepsilon)} \leq c\varepsilon^{1/2}\|\eta\|_{L^2(\Gamma)}.
  \end{align}
\end{lemma}

Next we give the Sobolev inequalities on $\Omega_\varepsilon$ with constants explicitly depending on $\varepsilon$.
For a function $\varphi$ on $\Omega_\varepsilon$, we write $\partial_n\varphi=(\bar{n}\cdot\nabla)\varphi$.
Note that $\partial_n\bar{\eta}=0$ in $\Omega_\varepsilon$ for the constant extension $\bar{\eta}$ of a function $\eta$ on $\Gamma$.

\begin{lemma} \label{L:Sob_CTD}
  Let $\varphi\in H^1(\Omega_\varepsilon)$.
  Then
  \begin{align} \label{E:Sob_L6}
    \|\varphi\|_{L^6(\Omega_\varepsilon)} \leq c\varepsilon^{-1/3}\|\varphi\|_{H^1(\Omega_\varepsilon)}^{2/3}\Bigl(\|\varphi\|_{L^2(\Omega_\varepsilon)}+\varepsilon\|\partial_n\varphi\|_{L^2(\Omega_\varepsilon)}\Bigr)^{1/3}.
  \end{align}
  Moreover, for all $p\in[2,6]$ we have
  \begin{align} \label{E:Sob_Lp}
    \|\varphi\|_{L^p(\Omega_\varepsilon)} \leq c_p\varepsilon^{-(p-2)/2p}\|\varphi\|_{H^1(\Omega_\varepsilon)},
  \end{align}
  where $c_p>0$ is a constant depending on $p$ but independent of $\varepsilon$.
\end{lemma}

The proof of Lemma \ref{L:Sob_CTD} is given in Section \ref{S:Pf_Sob}.

We also use the next result on the duality product of bulk and surface vector fields.

\begin{lemma}
  For all $v\in L^2(\Gamma,T\Gamma)$ and $u\in H^1(\Omega_\varepsilon)^3$ we have
  \begin{align} \label{E:Dual_CTD}
    \Bigl|(\bar{v},u)_{L^2(\Omega_\varepsilon)}\Bigr| \leq c\varepsilon^{1/2}\|v\|_{H^{-1}(\Gamma,T\Gamma)}\|u\|_{H^1(\Omega_\varepsilon)}.
  \end{align}
\end{lemma}

\begin{proof}
  We refer to \cite[Lemma 8.3]{Miu21_02} for the proof.
\end{proof}

The following lemmas is crucial for the proof of Theorem \ref{T:Er_CE}.

\begin{lemma} \label{L:Pdn_CTD}
  Suppose that the inequalities \eqref{E:Fr_Up} hold.
  Then
  \begin{align} \label{E:Pdn_CTD}
    \Bigl\|\overline{P}\partial_nu+\overline{W}u\Bigr\|_{L^2(\Omega_\varepsilon)} \leq c\varepsilon\|u\|_{H^2(\Omega_\varepsilon)}
  \end{align}
  for all $u\in H^2(\Omega_\varepsilon)^3$ satisfying the slip boundary conditions
  \begin{align} \label{E:Slip_Bo}
    u\cdot n_\varepsilon = 0, \quad 2\nu P_\varepsilon D(u)n_\varepsilon+\gamma_\varepsilon u = 0
  \end{align}
  on $\Gamma_\varepsilon^0$ or on $\Gamma_\varepsilon^1$.
\end{lemma}

\begin{proof}
  We refer to \cite[Lemma 4.7]{Miu21_02} for the proof.
\end{proof}

\subsection{Average operators} \label{SS:Average}
To analyze thin domain problems, it is useful to take the average of a function in the thin direction.
We define the average operator $M$ by
\begin{align*}
  M\varphi(y) = \frac{1}{\varepsilon g(y)}\int_{\varepsilon g_0(y)}^{\varepsilon g_1(y)}\varphi(y+rn(y))\,dr, \quad y\in\Gamma
\end{align*}
for a function $\varphi$ on $\Omega_\varepsilon$.
Also, we write $M_\tau u=PMu$ for the averaged tangential component of a vector field $u$ on $\Omega_\varepsilon$.
The following results play an important role in Section \ref{S:Diff_Est}.

\begin{lemma} \label{L:Ave_Hk}
  Let $u\in H^k(\Omega_\varepsilon)^3$ with $k=0,1$ and $H^0=L^2$.
  Then
  \begin{align} \label{E:Ave_Hk}
    \|M_\tau u\|_{H^k(\Gamma)} \leq c\varepsilon^{-1/2}\|u\|_{H^k(\Omega_\varepsilon)}
  \end{align}
\end{lemma}

\begin{proof}
  We refer to \cite[Lemmas 6.3 and 6.11]{Miu21_02} for the proof.
\end{proof}

\begin{lemma} \label{L:Ave_No}
  Let $u\in H^1(\Omega_\varepsilon)^3$ satisfy $u\cdot n_\varepsilon=0$ on $\Gamma_\varepsilon^0$ or on $\Gamma_\varepsilon^1$.
  Then
  \begin{align} \label{E:Ave_No}
    \|Mu\cdot n\|_{L^2(\Gamma)} &\leq c\varepsilon^{1/2}\|u\|_{H^1(\Omega_\varepsilon)}.
  \end{align}
  If in addition $u\in H^2(\Omega_\varepsilon)^3$, then
  \begin{align} \label{E:AvNo_H1}
    \|Mu\cdot n\|_{H^1(\Gamma)} \leq c\varepsilon^{1/2}\|u\|_{H^2(\Omega_\varepsilon)}.
  \end{align}
\end{lemma}

\begin{proof}
  We refer to \cite[Lemmas 6.5 and 6.13]{Miu21_02} for the proof.
\end{proof}

\begin{lemma} \label{L:Ave_Diff}
  Let $u\in H^1(\Omega_\varepsilon)^3$ satisfy $u\cdot n_\varepsilon=0$ on $\Gamma_\varepsilon^0$ or on $\Gamma_\varepsilon^1$.
  Then
  \begin{align} \label{E:Ave_Diff}
    \Bigl\|u-\overline{M_\tau u}\Bigr\|_{L^2(\Omega_\varepsilon)} \leq c\varepsilon\|u\|_{H^1(\Omega_\varepsilon)}.
  \end{align}
  If in addition $u\in H^2(\Omega_\varepsilon)^3$, then
  \begin{align} \label{E:AvDi_H1}
    \Bigl\|\overline{P}\nabla u-\overline{\nabla_\Gamma M_\tau u}\Bigr\|_{L^2(\Omega_\varepsilon)} \leq c\varepsilon\|u\|_{H^2(\Omega_\varepsilon)}.
  \end{align}
\end{lemma}

\begin{proof}
  We refer to \cite[Lemma 6.6]{Miu21_02} for the proof of \eqref{E:Ave_Diff}.
  Suppose that $u$ is in $H^2(\Omega_\varepsilon)^3$ and satisfies $u\cdot n_\varepsilon=0$ on $\Gamma_\varepsilon^0$ or on $\Gamma_\varepsilon^1$.
  It is shown in \cite[Lemma 6.12]{Miu21_02} that
  \begin{align*}
    \Bigl\|\overline{P}\nabla u-\overline{\nabla_\Gamma Mu}\Bigr\|_{L^2(\Omega_\varepsilon)} \leq c\varepsilon\|u\|_{H^2(\Omega_\varepsilon)}.
  \end{align*}
  Moreover, since $Mu=M_\tau u+(Mu\cdot n)n$ on $\Gamma$, we have
  \begin{align*}
    \|\nabla_\Gamma Mu-\nabla_\Gamma M_\tau u\|_{L^2(\Gamma)} \leq c\|Mu\cdot n\|_{H^1(\Gamma)} \leq c\varepsilon^{1/2}\|u\|_{H^2(\Omega_\varepsilon)}
  \end{align*}
  by \eqref{E:AvNo_H1}.
  By this inequality and \eqref{E:CE_L2CTD}, we find that
  \begin{align*}
    \Bigl\|\overline{\nabla_\Gamma Mu}-\overline{\nabla_\Gamma M_\tau u}\Bigr\|_{L^2(\Omega_\varepsilon)} \leq c\varepsilon^{1/2}\|\nabla_\Gamma Mu-\nabla_\Gamma M_\tau u\|_{L^2(\Gamma)} \leq c\varepsilon\|u\|_{H^2(\Omega_\varepsilon)}.
  \end{align*}
  Combining the above inequalities, we obtain \eqref{E:AvDi_H1}.
\end{proof}

\begin{lemma} \label{L:Ave_Asym}
  For all $u_1,u_2\in L^2(\Omega_\varepsilon)^3$, we have
  \begin{align} \label{E:Ave_Asym}
    \left|\Bigl(\overline{M_\tau u_1},u_2\Bigr)_{L^2(\Omega_\varepsilon)}-\Bigl(u_1,\overline{M_\tau u_2}\Bigr)_{L^2(\Omega_\varepsilon)}\right| \leq c\varepsilon\|u_1\|_{L^2(\Omega_\varepsilon)}\|u_2\|_{L^2(\Omega_\varepsilon)}.
  \end{align}
\end{lemma}

\begin{proof}
  We refer to \cite[Lemma 6.7]{Miu21_02} for the proof.
\end{proof}

\begin{lemma} \label{L:Ave_Div}
  Let $u\in H^1(\Omega_\varepsilon)^3$ satisfy $\mathrm{div}\,u=0$ in $\Omega_\varepsilon$ and $u\cdot n_\varepsilon=0$ on $\Gamma_\varepsilon$.
  Then
  \begin{align} \label{E:Ave_Div}
    \|\mathrm{div}_\Gamma(gM_\tau u)\|_{L^2(\Gamma)} \leq c\varepsilon^{1/2}\|u\|_{H^1(\Omega_\varepsilon)}.
  \end{align}
  If in addition $u\in H^2(\Omega_\varepsilon)^3$, then
  \begin{align} \label{E:ndn_CTD}
    \left\|\partial_nu\cdot\bar{n}-\frac{1}{\bar{g}}\overline{M_\tau u}\cdot\overline{\nabla_\Gamma g}\right\|_{L^2(\Omega_\varepsilon)} \leq c\varepsilon\|u\|_{H^2(\Omega_\varepsilon)}.
  \end{align}
\end{lemma}

\begin{proof}
  We refer to \cite[Lemmas 6.15 and 6.16]{Miu21_02} for the proof.
\end{proof}

\begin{lemma} \label{L:AT_L2In}
  For all $u\in L^2(\Omega_\varepsilon)^3$ and $\eta\in L^2(\Gamma,T\Gamma)$, we have
  \begin{align} \label{E:AT_L2In}
    \Bigl|(u,\bar{\eta})_{L^2(\Omega_\varepsilon)}-\varepsilon(gM_\tau u,\eta)_{L^2(\Gamma)}\Bigr| \leq c\varepsilon^{3/2}\|u\|_{L^2(\Omega_\varepsilon)}\|\eta\|_{L^2(\Gamma)}.
  \end{align}
\end{lemma}

\begin{proof}
  Using \eqref{E:CoV_CTD}--\eqref{E:CE_L2CTD} and H\"{o}lder's inequality, we can easily get
  \begin{align*}
    \Bigl|(u,\bar{\eta})_{L^2(\Omega_\varepsilon)}-\varepsilon(gMu,\eta)_{L^2(\Gamma)}\Bigr| \leq c\varepsilon^{3/2}\|u\|_{L^2(\Omega_\varepsilon)}\|\eta\|_{L^2(\Gamma)},
  \end{align*}
  see \cite[Lemma 5.23]{Miu20_03} for details.
  Moreover, $Mu\cdot\eta=M_\tau u\cdot\eta$ on $\Gamma$ since $\eta$ is tangential on $\Gamma$.
  Hence we obtain \eqref{E:AT_L2In} by these results.
\end{proof}

\begin{lemma} \label{L:Ave_RPe}
  Let $u\in H^1(\Omega_\varepsilon)^3$.
  Also, let $\mathcal{R}$ be given by \eqref{E:Def_Rota} and $w\in\mathcal{R}$ be of the form $w(x)=a\times x+b$, $x\in\mathbb{R}^3$ with $a,b\in\mathbb{R}^3$.
  If $(u,w)_{L^2(\Omega_\varepsilon)}=0$, then
  \begin{align} \label{E:Ave_RPe}
    \Bigl|(gM_\tau u,w)_{L^2(\Gamma)}\Bigr| \leq c\varepsilon^{1/2}(|a|+|b|)\|u\|_{L^2(\Omega_\varepsilon)},
  \end{align}
  where $c>0$ is a constant independent of $\varepsilon$, $u$, $a$, and $b$.
\end{lemma}

\begin{proof}
  We apply \eqref{E:CoV_CTD} to $(u,w)_{L^2(\Omega_\varepsilon)}=0$ to get
  \begin{align*}
    0 &= \int_\Gamma\int_{\varepsilon g_0(y)}^{\varepsilon g_1(y)}\{u(y+rn(y))\cdot w(y+rn(y))\}J(y,r)\,dr\,d\mathcal{H}^2(y) \\
    &= J_1+J_2+\int_\Gamma\int_{\varepsilon g_0(y)}^{\varepsilon g_1(y)}u(y+rn(y))\cdot w(y)\,dr\,d\mathcal{H}^2(y),
  \end{align*}
  where
  \begin{align*}
    J_1 &= \int_\Gamma\int_{\varepsilon g_0(y)}^{\varepsilon g_1(y)}[u(y+rn(y))\cdot\{w(y+rn(y))-w(y)\}]J(y,r)\,dr\,d\mathcal{H}^2(y), \\
    J_2 &= \int_\Gamma\int_{\varepsilon g_0(y)}^{\varepsilon g_1(y)}\{u(y+rn(y))\cdot w(y)\}\{J(y,r)-1\}\,dr\,d\mathcal{H}^2(y).
  \end{align*}
  Since $w(y)$ is independent of $r$, it follows that
  \begin{align*}
    \int_\Gamma\int_{\varepsilon g_0(y)}^{\varepsilon g_1(y)}u(y+rn(y))\cdot w(y)\,dr\,d\mathcal{H}^2(y) = \varepsilon\int_\Gamma g(y)Mu(y)\cdot w(y)\,d\mathcal{H}^2(y).
  \end{align*}
  Moreover, since $w\cdot n=0$ on $\Gamma$ by $w\in\mathcal{R}$, we have $Mu\cdot w=M_\tau u\cdot w$ on $\Gamma$ and thus
  \begin{align*}
    \int_\Gamma\int_{\varepsilon g_0(y)}^{\varepsilon g_1(y)}u(y+rn(y))\cdot w(y)\,dr\,d\mathcal{H}^2(y) = \varepsilon(gM_\tau u,w)_{L^2(\Gamma)}.
  \end{align*}
  To estimate $J_1$ and $J_2$, we see that $w$ is of the form $w(x)=a\times x+b$ and thus
  \begin{align*}
    |w(y+rn(y))-w(y)| = |r||a\times n(y)| \leq c\varepsilon|a|, \quad |w(y)| \leq c(|a|+|b|)
  \end{align*}
  for all $y\in\Gamma$ and $r\in[\varepsilon g_0(y),\varepsilon g_1(y)]$, since $\Gamma$ is compact and $g_0$ and $g_1$ are bounded on $\Gamma$.
  By these inequalities, \eqref{E:Jacob}, H\"{o}lder's inequality, and \eqref{E:CoV_L2}, we find that
  \begin{align*}
    |J_k| \leq c\varepsilon(|a|+|b|)\int_\Gamma\int_{\varepsilon g_0(y)}^{\varepsilon g_1(y)}|u(y+rn(y))|\,dr\,d\mathcal{H}^2(y) \leq c\varepsilon^{3/2}(|a|+|b|)\|u\|_{L^2(\Omega_\varepsilon)}
  \end{align*}
  for $k=1,2$.
  Combining the above results, we obtain
  \begin{align*}
    \Bigl|(gM_\tau u,w)_{L^2(\Gamma)}\Bigr| = \varepsilon^{-1}|J_1+J_2| \leq c\varepsilon^{1/2}(|a|+|b|)\|u\|_{L^2(\Omega_\varepsilon)}.
  \end{align*}
  Therefore, \eqref{E:Ave_RPe} is valid.
\end{proof}

\subsection{Extension of surface vector fields} \label{SS:Ex_SuVe}
In the study of the thin-film limit for \eqref{E:SNS_CTD}, we extend a vector field on $\Gamma$ to a suitable test function for a weak form of \eqref{E:SNS_CTD}.

Let $L_\sigma^2(\Omega_\varepsilon)$ and $H_{n,\sigma}^1(\Omega_\varepsilon)$ be the solenoidal spaces on $\Omega_\varepsilon$ given by \eqref{E:Sol_CTD}.
Also, let $\mathbb{L}_\varepsilon$ be the Helmholtz--Leray projection from $L^2(\Omega_\varepsilon)^3$ onto $L_\sigma^2(\Omega_\varepsilon)$.
It is known (see e.g. \cite{Tem01,Gal11,BoyFab13}) that the orthogonal complement of $L_\sigma^2(\Omega_\varepsilon)$ in $L^2(\Omega_\varepsilon)^3$ is of the form
\begin{align*}
  L_\sigma^2(\Omega_\varepsilon)^\perp = \{\nabla p\in L^2(\Omega_\varepsilon)^3 \mid p\in L^2(\Omega_\varepsilon)\}
\end{align*}
and that $\mathbb{L}_\varepsilon$ is given by $\mathbb{L}_\varepsilon u=u-\nabla\varphi$ for $u\in L^2(\Omega_\varepsilon)^3$, where $\varphi$ is a weak solution to
\begin{align*}
  \Delta\varphi = \mathrm{div}\,u \quad\text{in}\quad \Omega_\varepsilon, \quad \frac{\partial\varphi}{\partial n_\varepsilon} = u\cdot n_\varepsilon \quad\text{on}\quad \Gamma_\varepsilon.
\end{align*}
Moreover, if $u\in H^1(\Omega_\varepsilon)^3$, then $\mathbb{L}_\varepsilon u\in H_{n,\sigma}^1(\Omega_\varepsilon)$ by the elliptic regularity theory.

Let $\nabla_\Gamma$ and $W$ be the tangential gradient on $\Gamma$ and the Weingarten map of $\Gamma$ given in Section \ref{SS:Surf}.
We define
\begin{align*}
  \tau_\varepsilon^i(y) = \{I_3-\varepsilon g_i(y)W(y)\}^{-1}\nabla_\Gamma g_i(y), \quad y\in\Gamma, \, i=0,1.
\end{align*}
Then the unit outward normal vector field $n_\varepsilon$ of $\Gamma_\varepsilon$ is expressed by
\begin{align} \label{E:UON_CTD}
  n_\varepsilon(x) = \frac{(-1)^{i+1}}{\sqrt{1+\varepsilon^2|\bar{\tau}_\varepsilon^i(x)|^2}}\{\bar{n}(x)-\varepsilon\bar{\tau}_\varepsilon^i(x)\}, \quad x\in\Gamma_\varepsilon^i, \, i=0,1,
\end{align}
see \cite[Lemma 3.9]{Miu22_01}.
For a tangential vector field $v$ on $\Gamma$, we set
\begin{align*}
  E_\varepsilon v = \bar{v}+(\bar{v}\cdot\Psi_\varepsilon)\bar{n}, \quad \Psi_\varepsilon = \frac{1}{\bar{g}}\{(d-\varepsilon\bar{g}_0)\bar{\tau}_\varepsilon^1+(\varepsilon\bar{g}_1-d)\bar{\tau}_\varepsilon^0\} \quad\text{in}\quad N,
\end{align*}
where $d$ and $N$ are the signed distance from $\Gamma$ and the tubular neighborhood of $\Gamma$ given in Section \ref{SS:Surf}.
Then we easily observe by \eqref{E:UON_CTD} that $E_\varepsilon v\cdot n_\varepsilon=0$ on $\Gamma_\varepsilon$.
Moreover, $\mathrm{div}(E_\varepsilon v)$ is of order $\varepsilon$ in $\Omega_\varepsilon$ if $\mathrm{div}_\Gamma(gv)=0$ on $\Gamma$ (see \cite[Lemma 4.11]{Miu21_02}).

Combining $\mathbb{L}_\varepsilon$ and $E_\varepsilon$, we have a suitable extension of a vector field on $\Gamma$ as follows.

\begin{lemma} \label{L:ESV_L2}
  Let $\eta\in H^1(\Gamma,T\Gamma)$ satisfy $\mathrm{div}_\Gamma(g\eta)=0$ on $\Gamma$.
  Then
  \begin{align*}
    \eta_\varepsilon = \mathbb{L}_\varepsilon E_\varepsilon\eta \in H_{n,\sigma}^1(\Omega_\varepsilon)
  \end{align*}
  and there exists a constant $c>0$ independent of $\varepsilon$ and $\eta$ such that
  \begin{align} \label{E:ESV_L2}
    \|\eta_\varepsilon-\bar{\eta}\|_{L^2(\Omega_\varepsilon)} \leq c\varepsilon^{3/2}\|\eta\|_{H^1(\Gamma)}.
  \end{align}
\end{lemma}

\begin{proof}
  We refer to \cite[Lemma 7.3]{Miu20_03} for the proof.
\end{proof}

\section{Weak solution to the limit equations} \label{S:Sol_Lim}
The purpose of this section is to discuss the existence and uniqueness of a weak solution to the limit equations \eqref{E:Lim_Eq}.
We use the function spaces on $\Gamma$ given in Section \ref{SS:Surf}.

\subsection{Bilinear and trilinear forms} \label{SS:BL_Surf}
We give some properties of bilinear and trilinear forms appearing in a weak form of \eqref{E:Lim_Eq}.
Let $\nu>0$ and $\gamma^0,\gamma^1\geq0$ be the viscous and friction coefficients appearing in \eqref{E:Lim_Eq}, respectively.
We define
\begin{align} \label{E:BL_Surf}
  \begin{aligned}
    a_g(v_1,v_2) &= 2\nu\Bigl\{\bigl(gD_\Gamma(v_1),D_\Gamma(v_2)\bigr)_{L^2(\Gamma)}+(g^{-1}(v_1\cdot\nabla_\Gamma g),v_2\cdot\nabla_\Gamma g)_{L^2(\Gamma)}\Bigl\} \\
    &\qquad +(\gamma^0+\gamma^1)(v_1,v_2)_{L^2(\Gamma)}
  \end{aligned}
\end{align}
for $v_1,v_2\in H^1(\Gamma,T\Gamma)$ with $D_\Gamma(v_1)$ given by \eqref{E:Surf_SRT} and
\begin{align} \label{E:TL_Surf}
  b_g(v_1,v_2,v_3) = -(g(v_1\otimes v_2),\nabla_\Gamma v_3)_{L^2(\Gamma)}
\end{align}
for $v_1,v_2,v_3\in H^1(\Gamma,T\Gamma)$.
Note that $v_1\otimes v_2\in L^2(\Gamma)^{3\times 3}$ by \eqref{L:Lad_Ineq}.

\begin{lemma} \label{L:BLS_BoCo}
  For all $v_1,v_2\in H^1(\Gamma,T\Gamma)$, we have
  \begin{align} \label{E:BLS_Bou}
    |a_g(v_1,v_2)| \leq c\|v_1\|_{H^1(\Gamma)}\|v_2\|_{H^1(\Gamma)}.
  \end{align}
  Moreover, there exists a constant $c_a>0$ such that
  \begin{align} \label{E:BLS_Cor}
    \|v\|_{H^1(\Gamma)}^2 \leq c_aa_g(v,v)
  \end{align}
  for all $v\in H^1(\Gamma,T\Gamma)$ provided that $\gamma^0>0$ or $\gamma^1>0$ or $v$ satisfies \eqref{E:KoG_Per}.
\end{lemma}

\begin{proof}
  The inequality \eqref{E:BLS_Bou} follows from $|D_\Gamma(v_i)|\leq c|\nabla_\Gamma v_i|$ on $\Gamma$ for $i=1,2$ and $g\in C^1(\Gamma)$ and \eqref{E:G_Inf}.
  Also, \eqref{E:BLS_Cor} follows from \eqref{E:G_Inf} and Lemmas \ref{L:Korn_Surf} and \ref{L:Korn_G}.
\end{proof}

\begin{lemma} \label{L:TLS_BoAs}
  There exists a constant $c_b>0$ such that
  \begin{align} \label{E:TLS_Bou}
    |b_g(v_1,v_2,v_3)| \leq c_b\|v_1\|_{L^2(\Gamma)}^{1/2}\|v_1\|_{H^1(\Gamma)}^{1/2}\|v_2\|_{L^2(\Gamma)}^{1/2}\|v_2\|_{H^1(\Gamma)}^{1/2}\|v_3\|_{H^1(\Gamma)}
  \end{align}
  for all $v_1,v_2,v_3\in H^1(\Gamma,T\Gamma)$.
  Moreover,
  \begin{align} \label{E:TLS_Asym}
    b_g(v_1,v_2,v_3) = -b_g(v_1,v_3,v_2), \quad b_g(v_1,v_2,v_2) = 0
  \end{align}
  for all $v_1\in H_{g\sigma}^1(\Gamma,T\Gamma)$ and $v_2,v_3\in H^1(\Gamma,T\Gamma)$.
\end{lemma}

\begin{proof}
  The inequality \eqref{E:TLS_Bou} follows from H\"{o}lder's inequality and \eqref{E:Lad_Ineq}.
  Let us show \eqref{E:TLS_Asym}.
  We may assume that $v_2$ and $v_3$ are of class $C^1$ since $C^1(\Gamma,T\Gamma)$ is dense in $H^1(\Gamma,T\Gamma)$ (see \cite[Lemma 3.7]{Miu22_01}).
  Then we see by integration by parts \eqref{E:TD_IbP} that
  \begin{align*}
    b_g(v_1,v_2,v_3) = -b_g(v_1,v_3,v_2)+\int_\Gamma\{\mathrm{div}_\Gamma(gv_1)+gH(v_1\cdot n)\}(v_2\cdot v_3)\,d\mathcal{H}^2
  \end{align*}
  and the last integral vanishes by $v_1\in H_{g\sigma}^1(\Gamma,T\Gamma)$.
  Thus the first equality of \eqref{E:TLS_Asym} is valid.
  Also, we have the second equality of \eqref{E:TLS_Asym} by setting $v_2=v_3$ in the first one.
\end{proof}

Let us give some results related to the function space $\mathcal{K}_g(\Gamma)$ given by \eqref{E:Def_Kilg}.

\begin{lemma} \label{L:BL_Kg}
  If $\gamma^0=\gamma^1=0$, $v\in H^1(\Gamma,T\Gamma)$, and $w\in\mathcal{K}_g(\Gamma)$, then $a_g(v,w)=0$.
\end{lemma}

\begin{proof}
  The statement follows from $D_\Gamma(w)=0$ and $\nabla_\Gamma g\cdot w=0$ on $\Gamma$ and $\gamma^0=\gamma^1=0$.
\end{proof}

\begin{lemma} \label{L:TL_Kg}
  Let $v\in H^1(\Gamma,T\Gamma)$ and $w\in\mathcal{K}_g(\Gamma)$.
  Then
  \begin{align} \label{E:TKg_wvv}
    b_g(w,v,v) = 0.
  \end{align}
  For all $v_1,v_2\in H^1(\Gamma,T\Gamma)$ and $w\in\mathcal{K}_g(\Gamma)$, we have
  \begin{align} \label{E:TKg_v1v2}
    b_g(v_1,v_2,w) = -b_g(v_2,v_1,w), \quad b_g(v_1,v_1,w) = 0.
  \end{align}
\end{lemma}

\begin{proof}
  Let $w\in\mathcal{K}_g(\Gamma)$.
  Then $w\in H_{g\sigma}^1(\Gamma,T\Gamma)$ by Lemma \ref{L:Kilg_Sub} and thus \eqref{E:TKg_wvv} holds by \eqref{E:TLS_Asym}.
  Let $v_1,v_2\in H^1(\Gamma,T\Gamma)$.
  Then since $Pv_i=v_i$ for $i=1,2$ and $P^T=P$, we have
  \begin{align*}
    v_1\otimes v_2:\nabla_\Gamma w &= \{P(v_1\otimes v_2)P\}:\nabla_\Gamma w = v_1\otimes v_2:\{P(\nabla_\Gamma w)P\}, \\
    v_2\otimes v_1:\nabla_\Gamma w &= v_1\otimes v_2:(\nabla_\Gamma w)^T = v_1\otimes v_2:\{P(\nabla_\Gamma w)^TP\}
  \end{align*}
  on $\Gamma$.
  Hence $v_1\otimes v_2:\nabla_\Gamma w=-v_2\otimes v_1:\nabla_\Gamma w$ on $\Gamma$ by $D_\Gamma(w)=0$ and we get the first equality of \eqref{E:TKg_v1v2}.
  Also, the second equality follows from the first one with $v_2=v_1$.
\end{proof}

\subsection{Weak solution to the limit equations} \label{SS:Lim_Weak}
To give the definition of a weak solution to \eqref{E:Lim_Eq}, we recall the following result on the weighted Helmholtz--Leray decomposition on $\Gamma$ proved in \cite[Lemma 6.11]{Miu20_03}.
For a subset $\mathcal{X}$ of $L^2(\Gamma,T\Gamma)$, we write $\mathcal{X}^\perp$ for the orthogonal complement of $\mathcal{X}$ in $L^2(\Gamma,T\Gamma)$ with respect to the canonical inner product $(\cdot,\cdot)_{L^2(\Gamma)}$.

\begin{lemma} \label{L:Lgs_Per}
  We have $L_{g\sigma}^2(\Gamma,T\Gamma)^\perp = \{g\nabla_\Gamma q \in L^2(\Gamma,T\Gamma) \mid q\in H^1(\Gamma)\}$.
\end{lemma}

Based on this result and the integration by parts formula \eqref{E:TD_IbP}, we define a weak solution to \eqref{E:Lim_Eq} as follows.
Recall that we define $gf$ by \eqref{E:HinS_Mul} for $f\in H^{-1}(\Gamma,T\Gamma)$.

\begin{definition} \label{D:Lim_Weak}
  For a given $f\in H^{-1}(\Gamma,T\Gamma)$, we say that a tangential vector field $v$ on $\Gamma$ is a weak solution to \eqref{E:Lim_Eq} if $v\in H_{g\sigma}^1(\Gamma,T\Gamma)$ and
  \begin{align} \label{E:Lim_Weak}
    a_g(v,\eta)+b_g(v,v,\eta) = [gf,\eta]_{T\Gamma}
  \end{align}
  for all $\eta\in H_{g\sigma}^1(\Gamma,T\Gamma)$.
\end{definition}

For the existence of a weak solution, $f$ should satisfy the compatibility condition \eqref{E:Fsu_Comp}.

\begin{proposition} \label{P:LW_Comp}
  Let $f\in H^{-1}(\Gamma,T\Gamma)$.
  If $\gamma^0=\gamma^1=0$ and $\mathcal{K}_g(\Gamma)\neq\{0\}$, and if there exists a weak solution to \eqref{E:Lim_Eq}, then $[gf,w]_{T\Gamma}=0$ for all $w\in\mathcal{K}_g(\Gamma)$.
\end{proposition}

\begin{proof}
  Let $v$ be a weak solution to \eqref{E:Lim_Eq} and $w\in\mathcal{K}_g(\Gamma)$.
  Since $w\in H_{g\sigma}^1(\Gamma,T\Gamma)$ by Lemma \ref{L:Kilg_Sub}, and we can set $\eta=w$ in \eqref{E:Lim_Weak} and get $[gf,w]_{T\Gamma}=0$ by Lemmas \ref{L:BL_Kg} and \ref{L:TL_Kg}.
\end{proof}

As in the flat domain case, we can get the existence of a weak solution to \eqref{E:Lim_Eq}.

\begin{proposition} \label{P:LW_Exist}
  Let $f\in H^{-1}(\Gamma,T\Gamma)$ satisfy \eqref{E:Fsu_Comp}.
  Then there exists at least one weak solution $v$ to \eqref{E:Lim_Eq} such that $v\in\mathbb{V}_g$.
\end{proposition}

\begin{proof}
  The proof is the same as in the flat domain case under the slip boundary conditions \cite[Theorem 7.3]{AmrRej14}, so we just give the outline of the proof.

  Since $a_g$ is symmetric, bounded, and coercive on $\mathbb{V}_g$ by Lemma \ref{L:BLS_BoCo}, and since the embedding $\mathbb{V}_g\hookrightarrow\mathbb{H}_g$ is compact, we can take an orthonormal basis $\{v_k\}_{k=1}^\infty$ of $\mathbb{V}_g$ equipped with inner product $a_g(\cdot,\cdot)$ as in the flat domain case (see e.g. \cite[Theorem IV.5.5]{BoyFab13}).
  Using this basis, we can construct a vector field $v\in\mathbb{V}_g$ satisfying \eqref{E:Lim_Weak} for all $\eta\in\mathbb{V}_g$ by the Galerkin method with the aid of Lemmas \ref{L:BLS_BoCo} and \ref{L:TLS_BoAs} (see e.g. \cite[Theorems V.3.1]{BoyFab13} and \cite[Chapter II, Theorem 1.2]{Tem01} for the flat domain case).

  When $\gamma^0>0$ or $\gamma^1>0$ or $\mathcal{K}_g(\Gamma)=\{0\}$, this $v$ is a weak solution to \eqref{E:Lim_Eq} since $\mathbb{V}_g$ coincides with $H_{g\sigma}^1(\Gamma,T\Gamma)$.
  On the other hand, if $\gamma^0=\gamma^1=0$ and $\mathcal{K}_g(\Gamma)\neq\{0\}$, then $\mathbb{V}_g$ is strictly smaller than $H_{g\sigma}^1(\Gamma,T\Gamma)$.
  In this case, however, each test function $\eta\in H_{g\sigma}^1(\Gamma,T\Gamma)$ can be decomposed as $\eta=\zeta+w$ with $\zeta\in\mathbb{V}_g$ and $w\in\mathcal{K}_g(\Gamma)$ by Lemma \ref{L:Kilg_Sub}.
  Then $v$ satisfies \eqref{E:Lim_Weak} for the test function $\zeta\in\mathbb{V}_g$.
  Also, \eqref{E:Lim_Weak} holds for the test function $w\in\mathcal{K}_g(\Gamma)$ by \eqref{E:Fsu_Comp} and Lemmas \ref{L:BL_Kg} and \ref{L:TL_Kg}.
  Hence $v$ satisfies \eqref{E:Lim_Weak} for $\eta=\zeta+w$ and it is a weak solution to \eqref{E:Lim_Eq}.
\end{proof}

We also have the uniqueness of a weak solution to \eqref{E:Lim_Eq} in the class $\mathbb{V}_g$.

\begin{proposition} \label{P:LW_Uniq}
  For the constants $c_a$ and $c_b$ given in Lemmas \ref{L:BLS_BoCo} and \ref{L:TLS_BoAs}, let $\rho_u=(c_ac_b)^{-1}>0$.
  Also, let $f\in H^{-1}(\Gamma,T\Gamma)$ satisfy \eqref{E:Fsu_Comp}.
  If $v_1$ and $v_2$ are weak solutions to \eqref{E:Lim_Eq} such that $v_1,v_2\in\mathbb{V}_g$, and if $\|v_k\|_{H^1(\Gamma)}<\rho_u$ for $k=1$ or $k=2$, then $v_1=v_2$.
\end{proposition}

\begin{proof}
  We may assume $\|v_1\|_{H^1(\Gamma)}<\rho_u$ without loss of generality.
  Let $w=v_1-v_2$.
  Then since $v_1,v_2\in\mathbb{V}_g$ and they satisfy \eqref{E:Lim_Weak}, we see that $w\in\mathbb{V}_g$ and
  \begin{align*}
    a_g(w,\eta)+b_g(w,v_1,\eta)+b_g(v_2,w,\eta) = 0
  \end{align*}
  for all $\eta\in H_{g\sigma}^1(\Gamma,T\Gamma)$.
  Let $\eta=w$.
  Then we see by $v_2\in\mathbb{V}_g$ and \eqref{E:TLS_Asym} that
  \begin{align*}
    a_g(w,w)+b_g(w,v_1,w) = 0.
  \end{align*}
  Moreover, since $\gamma^0>0$ or $\gamma^1>0$ or $w\in\mathbb{V}_g$ satisfies \eqref{E:KoG_Per}, we can use \eqref{E:BLS_Cor} to get
  \begin{align*}
    \|w\|_{H^1(\Gamma)}^2 \leq c_aa_g(w,w) = -c_ab_g(w,v_1,w) \leq c_a|b_g(w,v_1,w)|.
  \end{align*}
  Thus, it follows from \eqref{E:TLS_Bou} and $\|v_1\|_{H^1(\Gamma)}<\rho_u=(c_ac_b)^{-1}$ that
  \begin{align*}
    \|w\|_{H^1(\Gamma)}^2 \leq c_ac_b\|v_1\|_{H^1(\Gamma)}\|w\|_{H^1(\Gamma)}^2, \quad 1-c_ac_b\|v_1\|_{H^1(\Gamma)} > 0,
  \end{align*}
  which implies $\|w\|_{H^1(\Gamma)}^2=0$ and thus $w=0$.
  Hence $v_1=v_2$.
\end{proof}

\begin{remark} \label{R:LW_Uniq}
  When $\gamma^0=\gamma^1=0$ and $\mathcal{K}_g(\Gamma)\neq\{0\}$, we can show that for each $w\in\mathcal{K}_g(\Gamma)$ there exists a vector field $v\in\mathbb{V}_g$ such that $v+w$ is a weak solution to \eqref{E:Lim_Eq} as in the proof of Proposition \ref{P:LW_Exist}.
  Thus, noting that $\mathcal{K}_g(\Gamma)$ is an infinite set and orthogonal to $\mathbb{H}_g$, we see that \eqref{E:Lim_Eq} admits infinitely many weak solutions in the class $H_{g\sigma}^1(\Gamma,T\Gamma)$.
  See also \cite[Remark 7.4]{AmrRej14} for the flat domain case.
  We also note that $\mathcal{K}_g(\Gamma)\neq\{0\}$ when $\Gamma$ is axially symmetric such as the unit sphere $S^2$ and when $g$ is constant on $\Gamma$.
\end{remark}

\section{Estimates for a solution to the bulk equations} \label{S:Bulk}
In this section we show explicit estimate in terms of $\varepsilon$ for a solution to the bulk Navier--Stokes equations \eqref{E:SNS_CTD}.
We also give some uniqueness results for a solution to \eqref{E:SNS_CTD}.

Throughout this section, we impose Assumptions \ref{Asmp_1} and \ref{Asmp_2} and use the function spaces given in Section \ref{S:Main}.
Also, let $\mathbb{P}_\varepsilon$ be the orthogonal projection from $L^2(\Omega_\varepsilon)^3$ onto $\mathcal{H}_\varepsilon$.

\subsection{Bilinear and trilinear forms} \label{SS:BTL_CTD}
We present some results on bilinear and trilinear forms appearing in a weak form of \eqref{E:SNS_CTD}.
Let $\nu>0$ be the viscosity coefficient independent of $\varepsilon$.
Also, let $\gamma_\varepsilon^0$ and $\gamma_\varepsilon^1$ be the friction coefficients appearing in \eqref{E:Fric}.
We define
\begin{align} \label{E:BL_CTD}
  a_\varepsilon(u_1,u_2) = 2\nu\bigl(D(u_1),D(u_2)\bigr)_{L^2(\Omega_\varepsilon)}+\gamma_\varepsilon^0(u_1,u_2)_{L^2(\Gamma_\varepsilon^0)}+\gamma_\varepsilon^1(u_1,u_2)_{L^2(\Gamma_\varepsilon^1)}
\end{align}
for $u_1,u_2\in H^1(\Omega_\varepsilon)^3$.
By integration by parts (see \cite[Lemma 7.1]{Miu22_01}), we easily observe that $a_\varepsilon$ is a bilinear form corresponding to the Stokes equations
\begin{align} \label{E:Sto_CTD}
  \left\{
  \begin{alignedat}{3}
    -\nu\Delta u^\varepsilon+\nabla p^\varepsilon = f^\varepsilon, \quad \mathrm{div}\,u^\varepsilon &= 0 &\quad &\text{in} &\quad &\Omega_\varepsilon, \\
    u^\varepsilon\cdot n_\varepsilon = 0, \quad 2\nu P_\varepsilon D(u^\varepsilon)n_\varepsilon+\gamma_\varepsilon u^\varepsilon &= 0, &\quad &\text{on} &\quad &\Gamma_\varepsilon.
  \end{alignedat}
  \right.
\end{align}
Clearly, $a_\varepsilon$ is symmetric.
Moreover, it is uniformly bounded and coercive on $\mathcal{V}_\varepsilon$.

\begin{lemma} \label{L:BL_Unif}
  Under Assumptions \ref{Asmp_1} and \ref{Asmp_2}, there exist constants $\varepsilon_0\in(0,1)$ and $c>0$ such that
  \begin{align} \label{E:BL_Unif}
    c^{-1}\|u\|_{H^1(\Omega_\varepsilon)}^2 \leq a_\varepsilon(u,u) \leq c\|u\|_{H^1(\Omega_\varepsilon)}^2
  \end{align}
  for all $\varepsilon\in(0,\varepsilon_0)$ and $u\in\mathcal{V}_\varepsilon$.
\end{lemma}

\begin{proof}
  We refer to \cite[Thoerem 2.4]{Miu22_01} for the proof.
\end{proof}

In what follows, we fix the constant $\varepsilon_0$ given in Lemma \ref{L:BL_Unif} and assume $\varepsilon\in(0,\varepsilon_0)$.
By the Lax--Milgram theory, $a_\varepsilon$ induces a bounded linear operator $A_\varepsilon$ from $\mathcal{V}_\varepsilon$ into its dual space.
We consider $A_\varepsilon$ an unbounded operator on $\mathcal{H}_\varepsilon$ with domain $D(A_\varepsilon)=\{u\in\mathcal{V}_\varepsilon \mid A_\varepsilon u\in\mathcal{H}_\varepsilon\}$ and call $A_\varepsilon$ the Stokes operator on $\mathcal{H}_\varepsilon$.
Then $A_\varepsilon$ is positive and self-adjoint on $\mathcal{H}_\varepsilon$ and thus $A_\varepsilon^{1/2}$ is well-defined on $\mathcal{H}_\varepsilon$ with domain $D(A_{\varepsilon}^{1/2})=\mathcal{V}_\varepsilon$.
Moreover, it follows from a regularity result for a solution to \eqref{E:Sto_CTD} (see \cite{SolSca73,Bei04,AmrRej14}) that
\begin{align*}
  D(A_\varepsilon) = \{u\in\mathcal{V}_\varepsilon\cap H^2(\Omega_\varepsilon)^3 \mid \text{$2\nu P_\varepsilon D(u_\varepsilon)n_\varepsilon+\gamma_\varepsilon u=0$ on $\Gamma_\varepsilon$}\}
\end{align*}
and $A_\varepsilon u=-\nu\mathbb{P}_\varepsilon\Delta u$ for $u\in D(A_\varepsilon)$.
Also, basic inequalities for $A_\varepsilon$ hold uniformly in $\varepsilon$.

\begin{lemma} \label{L:Ae_Equi}
  For all $\varepsilon\in(0,\varepsilon_0)$ and $u\in D(A_\varepsilon^{k/2})$ with $k=1,2$, we have
  \begin{align} \label{E:Ae_Equi}
    c^{-1}\|u\|_{H^k(\Omega_\varepsilon)} \leq \|A_\varepsilon^{k/2}u\|_{L^2(\Omega_\varepsilon)} \leq c\|u\|_{H^k(\Omega_\varepsilon)}.
  \end{align}
\end{lemma}

\begin{proof}
  We refer to \cite[Lemma 2.5 and Theorem 2.7]{Miu22_01} for the proof.
\end{proof}

For a trilinear term involving $A_\varepsilon$, we have an estimate similar to a 2D one.

\begin{lemma} \label{L:Tri_Est}
  There exist constants $d_1,d_2>0$ such that
  \begin{align} \label{E:Tri_Est}
    \begin{aligned}
      &\left|\bigl((u\cdot\nabla)u,A_\varepsilon u\bigr)_{L^2(\Omega_\varepsilon)}\right| \\
      &\qquad \leq \left(\frac{1}{4}+d_1\varepsilon^{1/2}\|A_\varepsilon^{1/2}u\|_{L^2(\Omega_\varepsilon)}\right)\|A_\varepsilon u\|_{L^2(\Omega_\varepsilon)}^2 \\
      &\qquad\qquad +d_2\Bigl(\|u\|_{L^2(\Omega_\varepsilon)}^2\|A_\varepsilon^{1/2}u\|_{L^2(\Omega_\varepsilon)}^4+\varepsilon^{-1}\|u\|_{L^2(\Omega_\varepsilon)}^2\|A_\varepsilon^{1/2}u\|_{L^2(\Omega_\varepsilon)}^2\Bigr)
    \end{aligned}
  \end{align}
  for all $\varepsilon\in(0,\varepsilon_0)$ and $u\in D(A_\varepsilon)$.
\end{lemma}

\begin{proof}
  We refer to \cite[Lemma 7.5]{Miu21_02} for the proof.
\end{proof}

Next we define a trilinear form $b_\varepsilon$ by
\begin{align} \label{E:TL_CTD}
  b_\varepsilon(u_1,u_2,u_3) = -(u_1\otimes u_2,\nabla u_3)_{L^2(\Omega_\varepsilon)} = -\int_{\Omega_\varepsilon}u_2\cdot[(u_1\cdot\nabla)u_3]\,dx
\end{align}
for $u_1,u_2,u_3\in H^1(\Omega_\varepsilon)^3$.
Note that $u_1\otimes u_2\in L^2(\Omega_\varepsilon)^{3\times3}$ by \eqref{E:Sob_Lp} with $p=4$.

\begin{lemma} \label{L:TLC_Asym}
  Let $u_1\in H_{n,\sigma}^1(\Omega_\varepsilon)$ and $u_2,u_3\in H^1(\Omega_\varepsilon)^3$.
  Then
  \begin{align} \label{E:TLC_Asym}
    b_\varepsilon(u_1,u_2,u_3) = -b_\varepsilon(u_1,u_3,u_2), \quad b_\varepsilon(u_1,u_2,u_2) = 0.
  \end{align}
\end{lemma}

\begin{proof}
  We have the first equality of \eqref{E:TLC_Asym} by integration by parts and by $\mathrm{div}\,u_1=0$ in $\Omega_\varepsilon$ and $u_1\cdot n_\varepsilon=0$ on $\Gamma_\varepsilon$.
  Also, the second equality follows from the first one with $u_3=u_2$.
\end{proof}

We also give some results related to the function space $\mathcal{R}_g$ given by \eqref{E:Def_Rota}.

\begin{lemma} \label{L:BTRg}
  Under the condition (A3) of Assumption \ref{Asmp_2}, we have
  \begin{align} \label{E:BTRg_H1}
    a_\varepsilon(u,w) = 0, \quad b_\varepsilon(w,u,u) = 0
  \end{align}
  for all $u\in H^1(\Omega_\varepsilon)^3$ and $w\in\mathcal{R}_g$.
  Moreover,
  \begin{align} \label{E:BTRg_Asym}
    b_\varepsilon(u_1,u_2,w) = -b_\varepsilon(u_2,u_1,w), \quad b_\varepsilon(u_1,u_1,w) = 0
  \end{align}
  for all $u_1,u_2\in H^1(\Omega_\varepsilon)^3$ and $w\in\mathcal{R}_g$.
\end{lemma}

\begin{proof}
  First we note that under the condition (A3) we have
  \begin{align} \label{Pf_BRg:Incl}
    \mathcal{R}_g = \mathcal{R}_0\cap\mathcal{R}_1 \subset H_{n,\sigma}^1(\Omega_\varepsilon)
  \end{align}
  by \cite[Lemma E.8]{Miu22_01}.
  Let $u\in H^1(\Omega_\varepsilon)^3$ and $w\in\mathcal{R}_g$.
  Since $w$ is of the form $w(x)=a\times x+b$ with $a,b\in\mathbb{R}^3$, we have $D(w)=0$ in $\Omega_\varepsilon$.
  By this fact and $\gamma_0^\varepsilon=\gamma_1^\varepsilon=0$ in the condition (A3), we have $a_\varepsilon(u,w)=0$.
  Also, since \eqref{Pf_BRg:Incl} holds, we have $b_\varepsilon(w,u,u)=0$ by \eqref{E:TLC_Asym}.

  Let $u_1,u_2\in H^1(\Omega_\varepsilon)^3$ and $w\in\mathcal{R}_g$.
  Since $D(w)=0$ in $\Omega_\varepsilon$, we have
  \begin{align*}
    u_1\otimes u_2:\nabla w = -u_1\otimes u_2:(\nabla w)^T = -u_2\otimes u_1:\nabla w \quad\text{in}\quad \Omega_\varepsilon.
  \end{align*}
  Hence we get the first equality of \eqref{E:BTRg_Asym} and the second one by setting $u_2=u_1$.
\end{proof}

\subsection{Solution to the bulk equations} \label{SS:Est_CTD}
Now we consider the bulk Navier--Stokes equations \eqref{E:SNS_CTD}.
First we give the definition of a weak solution to \eqref{E:SNS_CTD}.

\begin{definition} \label{D:NC_Weak}
  For a given $f^\varepsilon\in L^2(\Omega_\varepsilon)^3$, we say that a vector field $u^\varepsilon$ on $\Omega_\varepsilon$ is a weak solution to \eqref{E:SNS_CTD} if $u^\varepsilon\in H_{n,\sigma}^1(\Omega_\varepsilon)$ and
  \begin{align} \label{E:NC_Weak}
    a_\varepsilon(u^\varepsilon,\varphi)+b_\varepsilon(u^\varepsilon,u^\varepsilon,\varphi) = (f^\varepsilon,\varphi)_{L^2(\Omega_\varepsilon)}
  \end{align}
  for all $\varphi\in H_{n,\sigma}^1(\Omega_\varepsilon)$.
\end{definition}

As in Section \ref{SS:Lim_Weak}, we can show the following results by Lemmas \ref{L:BL_Unif}, \ref{L:TLC_Asym}, and \ref{L:BTRg}.

\begin{proposition} \label{P:F_Comp}
  Under the condition (A3) of Assumption \ref{Asmp_2}, let $f^\varepsilon\in L^2(\Omega_\varepsilon)^3$.
  If there exists a weak solution to \eqref{E:SNS_CTD}, then $f^\varepsilon$ is orthogonal to $\mathcal{R}_g$ in $L^2(\Omega_\varepsilon)^3$.
\end{proposition}

\begin{proposition} \label{P:NCW_Exist}
  Let $f^\varepsilon\in L^2(\Omega_\varepsilon)^3$ satisfy \eqref{E:Fth_Comp}.
  Then there exists at least one weak solution $u^\varepsilon$ to \eqref{E:SNS_CTD} such that $u^\varepsilon\in\mathcal{V}_\varepsilon$.
\end{proposition}

We also have a regularity result of a weak solution to \eqref{E:SNS_CTD}.

\begin{lemma} \label{L:NCW_Reg}
  Let $f^\varepsilon\in L^2(\Omega_\varepsilon)^3$ satisfy \eqref{E:Fth_Comp}.
  Then any weak solution to \eqref{E:SNS_CTD} belongs to $H^2(\Omega_\varepsilon)^3$.
\end{lemma}

\begin{proof}
  The statement can be shown by an $L^p$-regularity result for a solution to the Stokes equations \eqref{E:Sto_CTD} and a bootstrap argument.
  We refer to \cite[Theorem 7.5]{AmrRej14} for details.
\end{proof}

Let us give several estimates for a weak solution to \eqref{E:SNS_CTD} in the class $\mathcal{V}_\varepsilon$.

\begin{lemma} \label{L:NCW_H1}
  Let $c_1,c_2,\alpha,\beta>0$ be constants.
  Also, let $f^\varepsilon$ satisfy \eqref{E:Fth_Comp} and
  \begin{align} \label{E:NC_F_Est}
    \|\mathbb{P}_\varepsilon f^\varepsilon\|_{L^2(\Omega_\varepsilon)}^2 \leq c_1\varepsilon^{-1+\alpha}, \quad \|M_\tau\mathbb{P}_\varepsilon f^\varepsilon\|_{H^{-1}(\Gamma,T\Gamma)}^2 \leq c_2\varepsilon^{-1+\beta}.
  \end{align}
  Then there exists a constant $c_3>0$ independent of $\varepsilon$, $c_1$, $c_2$, $\alpha$, and $\beta$ such that
  \begin{align} \label{E:NCW_H1}
    \|u^\varepsilon\|_{H^1(\Omega_\varepsilon)}^2 \leq c_3(c_1\varepsilon^{1+\alpha}+c_2\varepsilon^\beta)
  \end{align}
  for any weak solution $u^\varepsilon$ to \eqref{E:SNS_CTD} satisfying $u^\varepsilon\in\mathcal{V}_\varepsilon$.
\end{lemma}

Recall that $Mu$ is the average in the thin direction of a vector field $u$ on $\Omega_\varepsilon$ and $M_\tau u$ is the tangential component of $Mu$ on $\Gamma$ (see Section \ref{SS:Average}).

\begin{proof}
  The proof is the same as in the nonstationary case (see (8.12) in \cite{Miu21_02}), but we give it for the completeness.
  Let $u^\varepsilon\in\mathcal{V}_\varepsilon$ be a weak solution to \eqref{E:SNS_CTD}.
  We take $\varphi=u^\varepsilon$ as a test function in \eqref{E:NC_Weak} and use the second equality of \eqref{E:TLC_Asym} with $u_1=u_2=u^\varepsilon$ to get
  \begin{align} \label{Pf_N1:Weak}
    a_\varepsilon(u^\varepsilon,u^\varepsilon) = (f^\varepsilon,u^\varepsilon)_{L^2(\Omega_\varepsilon)} = (\mathbb{P}_\varepsilon f^\varepsilon,u^\varepsilon)_{L^2(\Omega_\varepsilon)} = J_1+J_2,
  \end{align}
  where the second equality follows from $u^\varepsilon\in\mathcal{H}_\varepsilon$ and we set
  \begin{align*}
    J_1 = \Bigl(\mathbb{P}_\varepsilon f^\varepsilon,u^\varepsilon-\overline{M_\tau u^\varepsilon}\Bigr)_{L^2(\Omega_\varepsilon)}, \quad J_2 = \Bigl(\mathbb{P}_\varepsilon f^\varepsilon,\overline{M_\tau u^\varepsilon}\Bigr)_{L^2(\Omega_\varepsilon)}.
  \end{align*}
  Noting that $u^\varepsilon\in\mathcal{V}_\varepsilon$ satisfies $u^\varepsilon\cdot n_\varepsilon=0$ on $\Gamma_\varepsilon$, we see by \eqref{E:Ave_Diff} that
  \begin{align*}
    |J_1| \leq \|\mathbb{P}_\varepsilon f^\varepsilon\|_{L^2(\Omega_\varepsilon)}\Bigl\|u-\overline{M_\tau u^\varepsilon}\Bigr\|_{L^2(\Omega_\varepsilon)} \leq c\varepsilon\|f^\varepsilon\|_{L^2(\Omega_\varepsilon)}\|u\|_{H^1(\Omega_\varepsilon)}.
  \end{align*}
  Also, it follows from \eqref{E:Dual_CTD} and \eqref{E:Ave_Asym} that
  \begin{align*}
    |J_2| &\leq \left|\Bigl(\overline{M_\tau\mathbb{P}_\varepsilon f^\varepsilon},u^\varepsilon\Bigr)_{L^2(\Omega_\varepsilon)}\right|+\left|\Bigl(\mathbb{P}_\varepsilon f^\varepsilon,\overline{M_\tau u^\varepsilon}\Bigr)_{L^2(\Omega_\varepsilon)}-\Bigl(\overline{M_\tau\mathbb{P}_\varepsilon f^\varepsilon},u^\varepsilon\Bigr)_{L^2(\Omega_\varepsilon)}\right| \\
    &\leq c\left(\varepsilon^{1/2}\|M_\tau\mathbb{P}_\varepsilon f^\varepsilon\|_{H^{-1}(\Gamma,T\Gamma)}\|u^\varepsilon\|_{H^1(\Omega_\varepsilon)}+\varepsilon\|\mathbb{P}_\varepsilon f^\varepsilon\|_{L^2(\Omega_\varepsilon)}\|u^\varepsilon\|_{L^2(\Omega_\varepsilon)}\right).
  \end{align*}
  We apply these inequalities and \eqref{E:BL_Unif} (note that $u^\varepsilon\in\mathcal{V}_\varepsilon$) to \eqref{Pf_N1:Weak} to find that
  \begin{align*}
    \|u^\varepsilon\|_{H^1(\Omega_\varepsilon)}^2 &\leq c\|u^\varepsilon\|_{H^1(\Omega_\varepsilon)}\Bigl(\varepsilon\|f^\varepsilon\|_{L^2(\Omega_\varepsilon)}+\varepsilon^{1/2}\|M_\tau\mathbb{P}_\varepsilon f^\varepsilon\|_{H^{-1}(\Gamma,T\Gamma)}\Bigr) \\
    &\leq \frac{1}{2}\|u^\varepsilon\|_{H^1(\Omega_\varepsilon)}^2+c\Bigl(\varepsilon^2\|f^\varepsilon\|_{L^2(\Omega_\varepsilon)}^2+\varepsilon\|M_\tau\mathbb{P}_\varepsilon f^\varepsilon\|_{H^{-1}(\Gamma,T\Gamma)}^2\Bigr).
  \end{align*}
  Hence we get \eqref{E:NCW_H1} by this inequality and \eqref{E:NC_F_Est}.
\end{proof}

\begin{lemma} \label{L:NCW_H2}
  Let $c_1,c_2,\alpha,\beta>0$.
  There exist constants $\varepsilon_1\in(0,\varepsilon_0)$ and $c>0$ independent of $\varepsilon$ such that the following statement holds: if $\varepsilon\in(0,\varepsilon_1)$ and $f^\varepsilon$ satisfies \eqref{E:Fth_Comp} and \eqref{E:NC_F_Est}, and if $u^\varepsilon$ is a weak solution to \eqref{E:SNS_CTD} such that $u^\varepsilon\in\mathcal{V}_\varepsilon$, then $u^\varepsilon\in D(A_\varepsilon)$ and
  \begin{align} \label{E:NCW_H2}
    \|u^\varepsilon\|_{H^2(\Omega_\varepsilon)}^2 \leq c(\varepsilon^{-1+\alpha}+\varepsilon^{-1+2\beta}).
  \end{align}
\end{lemma}

\begin{proof}
  Let $\varepsilon\in(0,\varepsilon_0)$ and $u^\varepsilon\in\mathcal{V}_\varepsilon$ be a weak solution to \eqref{E:SNS_CTD}.
  Then $u^\varepsilon\in H^2(\Omega_\varepsilon)^3$ by Lemma \ref{L:NCW_Reg} and we see that $u^\varepsilon$ satisfies the strong form of \eqref{E:SNS_CTD} by the weak form \eqref{E:NC_Weak} and integration by parts.
  In particular, $u^\varepsilon\in D(A_\varepsilon)$.
  Let $\mathbb{P}_\varepsilon$ be the orthogonal projection from $L^2(\Omega_\varepsilon)^3$ onto $\mathcal{H}_\varepsilon$.
  We apply $\mathbb{P}_\varepsilon$ to \eqref{E:SNS_CTD}.
  Then we get
  \begin{align*}
    A_\varepsilon u^\varepsilon = -\nu\mathbb{P}_\varepsilon\Delta u^\varepsilon = \mathbb{P}_\varepsilon f^\varepsilon-\mathbb{P}_\varepsilon(u^\varepsilon\cdot\nabla)u^\varepsilon \quad\text{in}\quad \mathcal{H}_\varepsilon
  \end{align*}
  by $\mathbb{P}_\varepsilon\nabla p^\varepsilon=0$ (note that $\nabla p^\varepsilon\in L_\sigma^2(\Omega_\varepsilon)^\perp\subset\mathcal{H}_\varepsilon^\perp$).
  Noting that $A_\varepsilon u^\varepsilon\in\mathcal{H}_\varepsilon$, we take the $L^2(\Omega_\varepsilon)$-inner product of the above equality with $A_\varepsilon u^\varepsilon$ to find that
  \begin{align} \label{Pf_NCH2:Ip}
    \begin{aligned}
      \|A_\varepsilon u^\varepsilon\|_{L^2(\Omega_\varepsilon)}^2 &\leq \|\mathbb{P}_\varepsilon f^\varepsilon\|_{L^2(\Omega_\varepsilon)}\|A_\varepsilon u^\varepsilon\|_{L^2(\Omega_\varepsilon)}+\left|\bigl((u\cdot\nabla)u,A_\varepsilon u\bigr)_{L^2(\Omega_\varepsilon)}\right| \\
      &\leq \frac{1}{4}\|A_\varepsilon u^\varepsilon\|_{L^2(\Omega_\varepsilon)}^2+\|\mathbb{P}_\varepsilon f^\varepsilon\|_{L^2(\Omega_\varepsilon)}^2+\left|\bigl((u\cdot\nabla)u,A_\varepsilon u\bigr)_{L^2(\Omega_\varepsilon)}\right|.
    \end{aligned}
  \end{align}
  Moreover, since $u^\varepsilon\in\mathcal{V}_\varepsilon$ and $f^\varepsilon$ satisfies \eqref{E:NC_F_Est}, we see by \eqref{E:Ae_Equi} and \eqref{E:NCW_H1} that
  \begin{align*}
    \|A_\varepsilon^{1/2}u^\varepsilon\|_{L^2(\Omega_\varepsilon)}^2 \leq c\|u^\varepsilon\|_{H^1(\Omega_\varepsilon)}^2 \leq c(\varepsilon^{1+\alpha}+\varepsilon^\beta).
  \end{align*}
  Here and in what follows, $c$ denotes a general positive constant depending on $c_1$ and $c_2$ but independent of $\varepsilon$, $\alpha$, and $\beta$.
  From this inequality and \eqref{E:Tri_Est}, we deduce that
  \begin{align*}
    \left|\bigl((u\cdot\nabla)u,A_\varepsilon u\bigr)_{L^2(\Omega_\varepsilon)}\right| &\leq \left\{\frac{1}{4}+cd_1\varepsilon^{1/2}(\varepsilon^{1+\alpha}+\varepsilon^{\beta})^{1/2}\right\}\|A_\varepsilon u^\varepsilon\|_{L^2(\Omega_\varepsilon)}^2 \\
    &\qquad +cd_2\{(\varepsilon^{1+\alpha}+\varepsilon^{\beta})^3+\varepsilon^{-1}(\varepsilon^{1+\alpha}+\varepsilon^\beta)^2\}.
  \end{align*}
  Thus, taking $\varepsilon_1\in(0,\varepsilon_0)$ such that $cd_1\varepsilon_1^{1/2}(\varepsilon_1^{1+\alpha}+\varepsilon_1^{\beta})^{1/2}\leq1/4$ and noting that
  \begin{align*}
    (\varepsilon^{1+\alpha}+\varepsilon^{\beta})^3+\varepsilon^{-1}(\varepsilon^{1+\alpha}+\varepsilon^\beta)^2 \leq c(\varepsilon^{1+2\alpha}+\varepsilon^{-1+2\beta})
  \end{align*}
  by $0<\varepsilon<1$, $\alpha>0$, and $\beta>0$, we see that
  \begin{align*}
    \left|\bigl((u\cdot\nabla)u,A_\varepsilon u\bigr)_{L^2(\Omega_\varepsilon)}\right| \leq \frac{1}{2}\|A_\varepsilon u^\varepsilon\|_{L^2(\Omega_\varepsilon)}^2+c(\varepsilon^{1+2\alpha}+\varepsilon^{-1+2\beta})
  \end{align*}
  for $\varepsilon\in(0,\varepsilon_1)$.
  We apply this inequality and \eqref{E:NC_F_Est} to \eqref{Pf_NCH2:Ip} to get
  \begin{align*}
    \|A_\varepsilon u^\varepsilon\|_{L^2(\Omega_\varepsilon)}^2 \leq c(\varepsilon^{-1+\alpha}+\varepsilon^{-1+2\beta})
  \end{align*}
  by $0<\varepsilon<1$ and $\alpha>0$.
  By this equality and \eqref{E:Ae_Equi}, we obtain \eqref{E:NCW_H2}.
\end{proof}

By \eqref{E:AvNo_H1} and \eqref{E:NCW_H2}, we have the next result which shows that the normal component on $\Gamma$ of the average of $u^\varepsilon$ is sufficiently small as $\varepsilon\to0$.

\begin{proposition} \label{P:NCW_AN}
  Under the assumptions of Lemma \ref{L:NCW_H2}, we have
  \begin{align*}
    \|Mu^\varepsilon\cdot n\|_{H^1(\Gamma)}^2 \leq c(\varepsilon^\alpha+\varepsilon^{2\beta})
  \end{align*}
  with some constant $c>0$ independent of $\varepsilon$.
\end{proposition}

Lastly, let us give some uniqueness results for a weak solution to \eqref{E:SNS_CTD} in $\mathcal{V}_\varepsilon$.

\begin{proposition} \label{P:NCW_Uni}
  Let $c_1,\alpha>0$.
  There exist constants $\varepsilon_u\in(0,\varepsilon_0)$ and $c_u>0$ such that the following statement holds: if $\varepsilon\in(0,\varepsilon_u)$ and $f^\varepsilon$ satisfies \eqref{E:Fth_Comp} and
  \begin{align} \label{E:NCU_F}
    \|\mathbb{P}_\varepsilon f^\varepsilon\|_{L^2(\Omega_\varepsilon)}^2 \leq c_1\varepsilon^{-1+\alpha}, \quad \|M_\tau\mathbb{P}_\varepsilon f^\varepsilon\|_{H^{-1}(\Gamma,T\Gamma)}^2 \leq c_u,
  \end{align}
  and if $u_1^\varepsilon$ and $u_2^\varepsilon$ are weak solutions to \eqref{E:SNS_CTD} satisfying $u_1^\varepsilon,u_2^\varepsilon\in\mathcal{V}_\varepsilon$, then $u_1^\varepsilon=u_2^\varepsilon$.
\end{proposition}

\begin{proof}
  Let $c_u>0$ be a constant determined later and $f^\varepsilon$ satisfy \eqref{E:Fth_Comp} and \eqref{E:NCU_F}.
  Also, let $w^\varepsilon=u_1^\varepsilon-u_2^\varepsilon$.
  Since $u_1^\varepsilon,u_2^\varepsilon\in\mathcal{V}_\varepsilon$ and they satisfy \eqref{E:NC_Weak}, we have $w^\varepsilon\in\mathcal{V}_\varepsilon$ and
  \begin{align*}
    a_\varepsilon(w^\varepsilon,\varphi)+b_\varepsilon(w_\varepsilon,u_1^\varepsilon,\varphi)+b_\varepsilon(u_2^\varepsilon,w^\varepsilon,\varphi) = 0
  \end{align*}
  for all $\varphi\in H_{n,\sigma}^1(\Omega_\varepsilon)$.
  Let $\varphi=w^\varepsilon$ in this equality.
  Then
  \begin{align*}
    |b_\varepsilon(u_1^\varepsilon,w^\varepsilon,w^\varepsilon)| &\leq \|u_1^\varepsilon\|_{L^4(\Omega_\varepsilon)}\|w^\varepsilon\|_{L^4(\Omega_\varepsilon)}\|\nabla w^\varepsilon\|_{L^2(\Omega_\varepsilon)} \\
    &\leq c\varepsilon^{-1/2}\|u_1^\varepsilon\|_{H^1(\Omega_\varepsilon)}\|w^\varepsilon\|_{H^1(\Omega_\varepsilon)}^2
  \end{align*}
  by H\"{o}lder's inequality and \eqref{E:Sob_Lp} with $p=4$.
  Moreover, $b_\varepsilon(u_2^\varepsilon,w^\varepsilon,w^\varepsilon)=0$ by \eqref{E:TLC_Asym}.
  By these results and \eqref{E:BL_Unif} for $w^\varepsilon\in\mathcal{V}_\varepsilon$, we find that
  \begin{align*}
    \|w^\varepsilon\|_{H^1(\Omega_\varepsilon)}^2 \leq c_4\varepsilon^{-1/2}\|u_1^\varepsilon\|_{H^1(\Omega_\varepsilon)}\|w^\varepsilon\|_{H^1(\Omega_\varepsilon)}^2
  \end{align*}
  with a constant $c_4>0$ independent of $\varepsilon$, $u_1^\varepsilon$, and $u_2^\varepsilon$.
  Now we recall that $u_1^\varepsilon\in\mathcal{V}_\varepsilon$ and $f^\varepsilon$ satisfies \eqref{E:NCU_F} and thus \eqref{E:NC_F_Est} with $c_2=c_u$ and $\beta=1$.
  Hence
  \begin{align*}
    \|u^\varepsilon\|_{H^1(\Omega_\varepsilon)}^2 \leq c_3\varepsilon(c_1\varepsilon^\alpha+c_u)
  \end{align*}
  by \eqref{E:NCW_H1}, where $c_3>0$ is a constant independent of $\varepsilon$, $\alpha$, $c_1$, and $c_u$.
  Therefore,
  \begin{align*}
    \|w^\varepsilon\|_{H^1(\Omega_\varepsilon)}^2 \leq C(\varepsilon,c_u)\|w^\varepsilon\|_{H^1(\Omega_\varepsilon)}^2, \quad C(\varepsilon,c_u) = c_4c_3^{1/2}(c_1\varepsilon^\alpha+c_u)^{1/2},
  \end{align*}
  and if we take $\varepsilon_u,c_u>0$ such that $C(\varepsilon_u,c_u)=1$, then $C(\varepsilon,c_u)<1$ for $\varepsilon\in(0,\varepsilon_u)$ and it follows from the above inequality that $w^\varepsilon=0$, i.e. $u_1^\varepsilon=u_2^\varepsilon$.
\end{proof}

\begin{remark} \label{R:NCW_Uni}
  If $c_2\leq c_u$ for the constants $c_u$ in \eqref{P:NCW_Uni} and $c_2$ in the condition (a) of Theorem \ref{T:Error}, then Proposition \ref{P:NCW_Uni} implies that the weak solution $u^\varepsilon$ to \eqref{E:SNS_CTD} appearing in Theorem \ref{T:Error} is unique in the class $\mathcal{V}_\varepsilon$ when $\varepsilon$ is sufficiently small.
\end{remark}

\begin{proposition} \label{P:NCW_UBe}
  Let $c_1,c_2,\alpha>0$ and $\beta>1$.
  There exists a constant $\varepsilon_u'\in(0,\varepsilon_0)$ such that the following statement holds: if $\varepsilon\in(0,\varepsilon_u')$ and $f^\varepsilon$ satisfies \eqref{E:Fth_Comp} and \eqref{E:NC_F_Est}, and if $u_1^\varepsilon$ and $u_2^\varepsilon$ are weak solutions to \eqref{E:SNS_CTD} satisfying $u_1^\varepsilon,u_2^\varepsilon\in\mathcal{V}_\varepsilon$, then $u_1^\varepsilon=u_2^\varepsilon$.
\end{proposition}

\begin{proof}
  Let $\varepsilon_u$ and $c_u$ be the constants given in Proposition \ref{P:NCW_Uni}.
  Since $\beta>1$, we can take a constant $\varepsilon_\beta>0$ such that $c_2\varepsilon_\beta^{-1+\beta}=c_u$.
  Hence, if we set $\varepsilon_u'=\min\{\varepsilon_u,\varepsilon_\beta\}$, then the statement follows from Proposition \ref{P:NCW_Uni}.
\end{proof}

\begin{remark} \label{R:NCU_Be}
  If $\beta>1$ in \eqref{E:NC_F_Est}, then it follows from \eqref{E:Ave_Hk} and \eqref{E:NCW_H1} that
  \begin{align*}
    \|M_\tau u^\varepsilon\|_{H^1(\Gamma)}^2 \leq c\varepsilon^{-1}\|u^\varepsilon\|_{H^1(\Omega_\varepsilon)}^2 \leq c(\varepsilon^{\alpha}+\varepsilon^{\beta-1}) \to 0
  \end{align*}
  as $\varepsilon\to0$ for a (unique) weak solution $u^\varepsilon\in\mathcal{V}_\varepsilon$ to \eqref{E:SNS_CTD}.
\end{remark}

\section{Difference estimates for the bulk and surface velocities} \label{S:Diff_Est}
In this section we establish Theorems \ref{T:Error} and \ref{T:Er_CE}.
As in Section \ref{S:Bulk}, we impose Assumptions \ref{Asmp_1} and \ref{Asmp_2} and use the function spaces given in Section \ref{S:Main}.
Let $M$ and $M_\tau$ be the average operators given in Section \ref{SS:Average}.
We fix the constant $\varepsilon_0$ given in Lemma \ref{L:BL_Unif} and assume $\varepsilon\in(0,\varepsilon_0)$.
Let $A_\varepsilon$ be the Stokes operator on $\mathcal{H}_\varepsilon$ given in Section \ref{SS:BTL_CTD}.
Recall that $c$ denotes a general positive constant independent of $\varepsilon$.
Also, for a function $\eta$ on $\Gamma$, we write $\bar{\eta}$ for the constant extension of $\eta$ in the normal direction of $\Gamma$.

\subsection{Average of the bulk weak form} \label{SS:TF_AW}
Let $u^\varepsilon$ be a weak solution to \eqref{E:SNS_CTD}.
We derive a weak form on $\Gamma$ satisfied by $M_\tau u^\varepsilon$ from the weak form \eqref{E:NC_Weak} on $\Omega_\varepsilon$.

Let $a_g$ and $b_g$ be the bilinear and trilinear forms on $\Gamma$ given by \eqref{E:BL_Surf} and \eqref{E:TL_Surf}.
Also, let $a_\varepsilon$ and $b_\varepsilon$ be the bilinear and trilinear forms on $\Omega_\varepsilon$ given by \eqref{E:BL_CTD} and \eqref{E:TL_CTD}.
The following results proved in \cite{Miu20_03} show that $a_g$ and $b_g$ approximate $a_\varepsilon$ and $b_\varepsilon$, respectively.
The proofs are based on a careful analysis of integrals appearing in $a_\varepsilon$ and $b_\varepsilon$ with the aid of the change of variables formula \eqref{E:CoV_CTD} and by a suitable decomposition of a vector field on $\Omega_\varepsilon$ into an almost 2D average part and a sufficiently small residual part given in \cite{Miu21_02}.

Let $\mathbb{L}_\varepsilon$ be the Helmholtz--Leray projection from $L^2(\Omega_\varepsilon)$ onto $L_\sigma^2(\Omega_\varepsilon)$.
Also, let $E_\varepsilon$ be the extension operator of a vector field on $\Gamma$ given in Section \ref{SS:Ex_SuVe}.
Recall that $\eta_\varepsilon=\mathbb{L}_\varepsilon E_\varepsilon\eta$ is in $H_{n,\sigma}^1(\Omega_\varepsilon)$ and close to $\bar{\eta}$ if $\eta\in H_{g\sigma}^1(\Gamma,T\Gamma)$ (see Lemma \ref{L:ESV_L2}).

\begin{lemma} \label{L:BL_Appr}
  Let $u\in H^2(\Omega_\varepsilon)^3$ satisfy \eqref{E:Slip_Bo} on $\Gamma_\varepsilon$ and
  \begin{align*}
    \eta \in H_{g\sigma}^1(\Gamma,T\Gamma), \quad \eta_\varepsilon=\mathbb{L}_\varepsilon E_\varepsilon\eta \in H_{n,\sigma}^1(\Omega_\varepsilon).
  \end{align*}
  Then there exists a constant $c>0$ independent of $\varepsilon$, $u$, and $\eta$ such that
  \begin{align} \label{E:BL_Appr}
    |a_\varepsilon(u,\eta_\varepsilon)-\varepsilon a_g(M_\tau u,\eta)| \leq cR_\varepsilon^a(u)\|\eta\|_{H^1(\Gamma)},
  \end{align}
  where
  \begin{align} \label{E:BLA_Res}
    R_\varepsilon^a(u) = \varepsilon^{3/2}\|u\|_{H^2(\Omega_\varepsilon)}+\varepsilon^{1/2}\|u\|_{L^2(\Omega_\varepsilon)}\sum_{i=0,1}\left|\frac{\gamma_\varepsilon^i}{\varepsilon}-\gamma^i\right|.
  \end{align}
\end{lemma}

\begin{proof}
  We refer to \cite[Lemma 7.6]{Miu20_03} for the proof.
\end{proof}

\begin{lemma} \label{L:TL_Appr}
  Let $u_1\in H^2(\Omega_\varepsilon)^3$, $u_2\in H^1(\Omega_\varepsilon)^3$, and
  \begin{align*}
    \eta \in H_{g\sigma}^1(\Gamma,T\Gamma), \quad \eta_\varepsilon = \mathbb{L}_\varepsilon E_\varepsilon\eta \in H_{n,\sigma}^1(\Omega_\varepsilon).
  \end{align*}
  Suppose that $u_1$ satisfies $\mathrm{div}\,u_1=0$ in $\Omega_\varepsilon$ and \eqref{E:Slip_Bo} on $\Gamma_\varepsilon$ and that $u_2\cdot n_\varepsilon=0$ on $\Gamma_\varepsilon^0$ or on $\Gamma_\varepsilon^1$.
  Then there exists a constant $c>0$ independent of $\varepsilon$, $u_1$, $u_2$, and $\eta$ such that
  \begin{align} \label{E:TL_Appr}
    |b_\varepsilon(u_1,u_2,\eta_\varepsilon)-\varepsilon b_g(M_\tau u_1,M_\tau u_2,\eta)| \leq cR_\varepsilon^b(u_1,u_2)\|\eta\|_{H^1(\Gamma)},
  \end{align}
  where
  \begin{align} \label{E:TLA_Res}
    \begin{aligned}
      R_\varepsilon^b(u_1,u_2) &= \varepsilon^{3/2}\|u_1\otimes u_2\|_{L^2(\Omega_\varepsilon)}+\varepsilon\|u_1\|_{H^1(\Omega_\varepsilon)}\|u_2\|_{H^1(\Omega_\varepsilon)} \\
      &\qquad +\Bigl(\varepsilon\|u_1\|_{H^2(\Omega_\varepsilon)}+\varepsilon^{1/2}\|u_1\|_{L^2(\Omega_\varepsilon)}^{1/2}\|u_1\|_{H^2(\Omega_\varepsilon)}^{1/2}\Bigr)\|u_2\|_{L^2(\Omega_\varepsilon)}
    \end{aligned}
  \end{align}
\end{lemma}

\begin{proof}
  We refer to \cite[Lemma 7.7]{Miu20_03} for the proof.
\end{proof}

Using these results, we derive a weak form of $M_\tau u^\varepsilon$ for a weak solution $u^\varepsilon$ to \eqref{E:SNS_CTD}.

\begin{lemma} \label{L:Ave_Weak}
  Let $c_1,c_2>0$, $\alpha\in(0,1]$, $\beta=1$, and $\varepsilon_1$ be the constant given in Lemma \ref{L:NCW_H2}.
  For $\varepsilon\in(0,\varepsilon_1)$, let $f^\varepsilon$ satisfy \eqref{E:Fth_Comp} and \eqref{E:NC_F_Est}, and let $u^\varepsilon$ be a weak solution to \eqref{E:SNS_CTD} such that $u^\varepsilon\in\mathcal{V}_\varepsilon$.
  Then $u^\varepsilon\in D(A_\varepsilon)$ and
  \begin{align} \label{E:AW_H1H2}
    \|u^\varepsilon\|_{H^1(\Omega_\varepsilon)} \leq c\varepsilon^{1/2}, \quad \|u^\varepsilon\|_{H^2(\Omega_\varepsilon)} \leq c\varepsilon^{-1/2+\alpha/2}
  \end{align}
  with a constant $c>0$ independent of $\varepsilon$, $f^\varepsilon$, and $u^\varepsilon$.
  Moreover,
  \begin{align} \label{E:Ave_Weak}
    a_g(M_\tau u^\varepsilon,\eta)+b_g(M_\tau u^\varepsilon,M_\tau u^\varepsilon,\eta) = (gM_\tau\mathbb{P}_\varepsilon f^\varepsilon,\eta)_{L^2(\Gamma)}+R_\varepsilon^1(\eta)
  \end{align}
  for all $\eta\in H_{g\sigma}^1(\Gamma,T\Gamma)$, where $R_\varepsilon^1(\eta)$ is a residual term bounded by
  \begin{align} \label{E:AW_Res}
    |R_\varepsilon^1(\eta)| \leq c\delta(\varepsilon)\|\eta\|_{H^1(\Gamma)}, \quad \delta(\varepsilon) = \varepsilon^{\alpha/4}+\sum_{i=0,1}\left|\frac{\gamma_\varepsilon^i}{\varepsilon}-\gamma^i\right|
  \end{align}
  with a constant $c>0$ independent of $\varepsilon$, $f^\varepsilon$, $u^\varepsilon$, and $\eta$.
\end{lemma}

\begin{proof}
  Since $f^\varepsilon$ satisfies \eqref{E:Fth_Comp} and \eqref{E:NC_F_Est} with $\beta=1$ and $u^\varepsilon\in\mathcal{V}_\varepsilon$, it follows from Lemmas \ref{L:NCW_H1} and \ref{L:NCW_H2}, $0<\varepsilon<1$, and $\alpha\in(0,1]$ that $u^\varepsilon\in D(A_\varepsilon)$ and \eqref{E:AW_H1H2} holds.
  In particular, $u^\varepsilon$ satisfies $\mathrm{div}\,u^\varepsilon=0$ in $\Omega_\varepsilon$ and \eqref{E:Slip_Bo} on $\Gamma_\varepsilon$.

  For $\eta\in H_{g\sigma}^1(\Gamma,T\Gamma)$, let $\varphi=\eta_\varepsilon=\mathbb{L}_\varepsilon E_\varepsilon\eta\in H_{n,\sigma}^1(\Omega_\varepsilon)$ in \eqref{E:NC_Weak}.
  Then
  \begin{align*}
    a_\varepsilon(u^\varepsilon,\eta_\varepsilon)+b_\varepsilon(u^\varepsilon,u^\varepsilon,\eta_\varepsilon) = (f^\varepsilon,\eta_\varepsilon)_{L^2(\Omega_\varepsilon)} = (\mathbb{P}_\varepsilon f^\varepsilon,\eta_\varepsilon)_{L^2(\Omega_\varepsilon)},
  \end{align*}
  where the second equality holds since $f^\varepsilon$ satisfies \eqref{E:Fth_Comp} and $\mathbb{P}_\varepsilon$ is the orthogonal projection from $L^2(\Omega_\varepsilon)^3$ onto the space $\mathcal{H}_\varepsilon$ given by \eqref{E:Def_Heps}.
  We divide both sides of the above equality by $\varepsilon$ and replace each term of the resulting equality by the corresponding term of \eqref{E:Ave_Weak}.
  Then we have \eqref{E:Ave_Weak} with
  \begin{align} \label{Pf_AW:Res}
    R_\varepsilon^1(\eta) = \varepsilon^{-1}(-J_1-J_2+J_3),
  \end{align}
  where
  \begin{align*}
    J_1 &= a_\varepsilon(u^\varepsilon,\eta_\varepsilon)-\varepsilon a_g(M_\tau u^\varepsilon,\eta), \\
    J_2 &= b_\varepsilon(u^\varepsilon,u^\varepsilon,\eta_\varepsilon)-\varepsilon b_g(M_\tau u^\varepsilon,M_\tau u^\varepsilon,\eta), \\
    J_3 &= (\mathbb{P}_\varepsilon f^\varepsilon,\eta_\varepsilon)_{L^2(\Omega_\varepsilon)}-(gM_\tau\mathbb{P}_\varepsilon f^\varepsilon,\eta)_{L^2(\Gamma)}.
  \end{align*}
  Let us estimate $J_1$, $J_2$, and $J_3$.
  We see by \eqref{E:BL_Appr}, \eqref{E:BLA_Res}, and \eqref{E:AW_H1H2} that
  \begin{align*}
    |J_1| \leq cR_\varepsilon^a(u^\varepsilon)\|\eta\|_{H^1(\Gamma)} \leq c\varepsilon\left(\varepsilon^{\alpha/2}+\sum_{i=0,1}\left|\frac{\gamma_\varepsilon^i}{\varepsilon}-\gamma^i\right|\right)\|\eta\|_{H^1(\Gamma)}.
  \end{align*}
  Next it follows from H\"{o}lder's inequality, \eqref{E:Sob_Lp} with $p=4$, and \eqref{E:AW_H1H2} that
  \begin{align*}
    \|u^\varepsilon\otimes u^\varepsilon\|_{L^2(\Omega_\varepsilon)} \leq \|u^\varepsilon\|_{L^4(\Omega_\varepsilon)}^2 \leq c\varepsilon^{-1/2}\|u^\varepsilon\|_{H^1(\Omega_\varepsilon)}^2 \leq c\varepsilon^{1/2}.
  \end{align*}
  Hence we apply \eqref{E:TL_Appr} to $J_2$ and use the above estimate and \eqref{E:AW_H1H2} to get
  \begin{align*}
    |J_2| \leq cR_\varepsilon^b(u^\varepsilon,u^\varepsilon)\|\eta\|_{H^1(\Gamma)} \leq c\varepsilon(\varepsilon+\varepsilon^{\alpha/2}+\varepsilon^{\alpha/4})\|\eta\|_{H^1(\Gamma)}.
  \end{align*}
  We also observe by \eqref{E:AT_L2In}, \eqref{E:ESV_L2}, and \eqref{E:NC_F_Est} that
  \begin{align*}
    |J_3| &\leq \Bigl|(\mathbb{P}_\varepsilon f^\varepsilon,\eta_\varepsilon)_{L^2(\Omega_\varepsilon)}-(\mathbb{P}_\varepsilon f^\varepsilon,\bar{\eta})_{L^2(\Omega_\varepsilon)}\Bigr|+\Bigl|(\mathbb{P}_\varepsilon f^\varepsilon,\bar{\eta})_{L^2(\Omega_\varepsilon)}-(gM_\tau\mathbb{P}_\varepsilon f^\varepsilon,\eta)_{L^2(\Gamma)}\Bigr| \\
    &\leq \|\mathbb{P}_\varepsilon f^\varepsilon\|_{L^2(\Omega_\varepsilon)}\|\eta_\varepsilon-\bar{\eta}\|_{L^2(\Omega_\varepsilon)}+c\varepsilon^{3/2}\|\mathbb{P}_\varepsilon f^\varepsilon\|_{L^2(\Omega_\varepsilon)}\|\eta\|_{L^2(\Gamma)} \\
    &\leq c\varepsilon^{1+\alpha/2}\|\eta\|_{H^1(\Gamma)}.
  \end{align*}
  Applying these estimates to \eqref{Pf_AW:Res} and noting that $\varepsilon^\sigma\leq\varepsilon^{\alpha/4}$ for $\sigma=1,\alpha/2$ by $0<\varepsilon<1$ and $\alpha\in(0,1]$, we obtain \eqref{E:AW_Res}.
\end{proof}

\subsection{Modification of the averaged bulk velocity} \label{SS:Rep}
Let $u^\varepsilon\in\mathcal{V}_\varepsilon$ and $v\in\mathbb{V}_g$ be weak solutions to \eqref{E:SNS_CTD} and \eqref{E:Lim_Eq}, respectively.
To prove \eqref{E:Error}, we intend to take the difference of the weak forms \eqref{E:Lim_Weak} and \eqref{E:Ave_Weak}, test $M_\tau u^\varepsilon-v$, and apply Lemmas \ref{L:BLS_BoCo} and \ref{L:TLS_BoAs}.
However, since $M_\tau u^\varepsilon\not\in\mathbb{V}_g$ in general, we need to modify $M_\tau u^\varepsilon$ to an element of $\mathbb{V}_g$ in order to take a suitable test function and apply the coerciveness \eqref{E:BLS_Cor} of the bilinear form $a_g$.

\begin{lemma} \label{L:DgG_Surf}
  Let $\eta\in L^2(\Gamma)$ satisfy $\int_\Gamma\eta\,d\mathcal{H}^2=0$.
  Then the problem
  \begin{align} \label{E:DgG_Surf}
    -\mathrm{div}_\Gamma(g\nabla_\Gamma q) = \eta \quad\text{on}\quad \Gamma, \quad \int_\Gamma q\,d\mathcal{H}^2 = 0
  \end{align}
  admits a unique solution $q\in H^2(\Gamma)$.
  Moreover, the inequality
  \begin{align} \label{E:DgG_Est}
    \|q\|_{H^2(\Gamma)} \leq c\|\eta\|_{L^2(\Gamma)}
  \end{align}
  holds with a constant $c>0$ independent of $\eta$ and $q$.
\end{lemma}

Note that the condition $\int_\Gamma\eta\,d\mathcal{H}^2=0$ is necessary for the existence of a solution to \eqref{E:DgG_Surf}.
Indeed, if there exists a solution $q\in H^2(\Gamma)$ to \eqref{E:DgG_Surf}, then
\begin{align*}
  \int_\Gamma\eta\,d\mathcal{H}^2 = -\int_\Gamma\mathrm{div}_\Gamma(g\nabla_\Gamma q)\,d\mathcal{H}^2 = \int_\Gamma gH(\nabla_\Gamma q\cdot n)\,d\mathcal{H}^2 = 0
\end{align*}
by \eqref{E:TD_IbP} and $\nabla_\Gamma q\cdot n=0$ on $\Gamma$.

\begin{proof}
  For all $q\in H^1(\Gamma)$ satisfying $\int_\Gamma q\,d\mathcal{H}^2=0$, we have
  \begin{align*}
    \|q\|_{H^1(\Gamma)} \leq c\|\nabla_\Gamma q\|_{L^2(\Gamma)} \leq c\|g^{1/2}\nabla_\Gamma q\|_{L^2(\Gamma)}
  \end{align*}
  by Poincar\'{e}'s inequality on $\Gamma$ (see e.g. \cite[Lemma 3.8]{Miu20_03}) and \eqref{E:G_Inf}.
  By this inequality and the Lax--Milgram theorem, we find that the problem \eqref{E:DgG_Surf} admits a unique weak solution $q\in H^1(\Gamma)$ in the sense that
  \begin{align} \label{Pf_DGS:Weak}
    (g\nabla_\Gamma q,\nabla_\Gamma\xi)_{L^2(\Gamma)} = (\eta,\xi)_{L^2(\Gamma)} \quad\text{for all}\quad \xi\in H^1(\Gamma)
  \end{align}
  and that there exists a constant $c>0$ such that
  \begin{align} \label{Pf_DGS:H1}
    \|q\|_{H^1(\Gamma)} \leq c\|\eta\|_{L^2(\Gamma)}.
  \end{align}
  Moreover, replacing $\xi$ by $g^{-1}\xi$ in \eqref{Pf_DGS:Weak}, we have
  \begin{align} \label{Pf_DGS:Poi}
    (\nabla_\Gamma q,\nabla_\Gamma\xi)_{L^2(\Gamma)} = (\zeta,\xi)_{L^2(\Gamma)}, \quad \zeta = g^{-1}(\eta+\nabla_\Gamma q\cdot\nabla_\Gamma g) \in L^2(\Gamma)
  \end{align}
  for all $\xi\in H^1(\Gamma)$, which means that $q\in H^1(\Gamma)$ is a (unique) weak solution to
  \begin{align*}
    -\Delta_\Gamma q = \zeta \quad\text{on}\quad \Gamma, \quad \int_\Gamma q\,d\mathcal{H}^2 = 0.
  \end{align*}
  Note that $\int_\Gamma\zeta\,d\mathcal{H}^2=0$ by \eqref{Pf_DGS:Poi} with $\xi\equiv1$.
  Hence $q\in H^2(\Gamma)$ and
  \begin{align*}
    \|q\|_{H^2(\Gamma)} \leq c\|\zeta\|_{L^2(\Gamma)} \leq c\Bigl(\|\eta\|_{L^2(\Gamma)}+\|q\|_{H^1(\Gamma)}\Bigr) \leq c\|\eta\|_{L^2(\Gamma)},
  \end{align*}
  i.e. \eqref{E:DgG_Est} holds by the elliptic regularity theorem for the Poisson equation on $\Gamma$ (see e.g. \cite[Theorem 3.3]{DziEll13}), \eqref{E:G_Inf}, $g\in C^4(\Gamma)$, and \eqref{Pf_DGS:H1}.
\end{proof}

\begin{lemma} \label{L:AveT_Vg}
  Let $c$ and $\sigma$ be positive constants.
  Suppose that $u^\varepsilon\in\mathcal{V}_\varepsilon$ and
  \begin{align} \label{E:ATVg_Bo}
    \|u^\varepsilon\|_{H^1(\Omega_\varepsilon)} \leq c\varepsilon^{-1/2+\sigma/2}.
  \end{align}
  When the condition (A1) of Assumption \ref{Asmp_2} is imposed, suppose further that $\gamma^0>0$ or $\gamma^1>0$.
  Then there exists a vector field $v_\varepsilon\in\mathbb{V}_g$ such that
  \begin{align} \label{E:ATVg_Diff}
    \|M_\tau u^\varepsilon-v_\varepsilon\|_{H^1(\Gamma)} \leq c'\varepsilon^{\sigma/2},
  \end{align}
  where $c'>0$ is a constant independent of $\varepsilon$ and $u^\varepsilon$.
\end{lemma}

\begin{proof}
  Let $u^\varepsilon\in\mathcal{V}_\varepsilon$ satisfy \eqref{E:ATVg_Bo}.
  Since $u^\varepsilon$ satisfies $\mathrm{div}\,u^\varepsilon=0$ in $\Omega_\varepsilon$ and $u^\varepsilon\cdot n_\varepsilon=0$ on $\Gamma_\varepsilon$, we can use \eqref{E:Ave_Div} and \eqref{E:ATVg_Bo} to get
  \begin{align} \label{Pf_ATV:div}
    \|\mathrm{div}_\Gamma(gM_\tau u^\varepsilon)\|_{L^2(\Gamma)} \leq c\varepsilon^{1/2}\|u^\varepsilon\|_{H^1(\Omega_\varepsilon)} \leq c\varepsilon^{\sigma/2}.
  \end{align}
  Let $\eta^\varepsilon=-\mathrm{div}_\Gamma(gM_\tau u^\varepsilon)$ on $\Gamma$.
  Since $\eta^\varepsilon\in L^2(\Gamma)$ and
  \begin{align*}
    \int_\Gamma\eta_\varepsilon\,d\mathcal{H}^2 = \int_\Gamma gH(M_\tau u^\varepsilon\cdot n)\,d\mathcal{H}^2 = 0
  \end{align*}
  by \eqref{E:TD_IbP} and $M_\tau u^\varepsilon\cdot n=0$ on $\Gamma$, there exists a unique solution $q^\varepsilon\in H^2(\Gamma)$ to \eqref{E:DgG_Surf} with source term $\eta^\varepsilon$ by Lemma \ref{L:DgG_Surf}.
  Moreover, we see by \eqref{E:DgG_Est} and \eqref{Pf_ATV:div} that
  \begin{align} \label{Pf_ATV:qe}
    \|q^\varepsilon\|_{H^2(\Gamma)} \leq c\|\eta^\varepsilon\|_{L^2(\Gamma)} = c\|\mathrm{div}_\Gamma(gM_\tau u^\varepsilon)\|_{L^2(\Gamma)} \leq c\varepsilon^{\sigma/2}.
  \end{align}
  Thus, if we set $v_1^\varepsilon=M_\tau u^\varepsilon-\nabla_\Gamma q^\varepsilon$ on $\Gamma$, then $v_1^\varepsilon\in H_{g\sigma}^1(\Gamma,T\Gamma)$ and
  \begin{align} \label{Pf_ATV:v1e}
    \|M_\tau u^\varepsilon-v_1^\varepsilon\|_{H^1(\Gamma)} = \|\nabla_\Gamma q^\varepsilon\|_{H^1(\Gamma)} \leq c\varepsilon^{\sigma/2}
  \end{align}
  by \eqref{Pf_ATV:qe}.
  When $\gamma^0>0$ or $\gamma^1>0$ or $\mathcal{K}_g(\Gamma)=\{0\}$, we define $v^\varepsilon=v_1^\varepsilon$ to get $v^\varepsilon\in\mathbb{V}_g$ and \eqref{E:ATVg_Diff} by the above result.

  Now suppose that $\gamma^0=\gamma^1=0$ and $\mathcal{K}_g(\Gamma)\neq\{0\}$.
  By the assumption of the lemma, this case occurs only when the condition (A3) of Assumption \ref{Asmp_2} is imposed.
  In this case, we need to take a component of $v_1^\varepsilon$ orthogonal to $\mathcal{K}_g(\Gamma)$ with respect to the weighted inner product $(g\,\cdot,\cdot)_{L^2(\Gamma)}$.
  Since $\mathcal{K}_g(\Gamma)$ is a finite dimensional subspace of $H_{g\sigma}^1(\Gamma,T\Gamma)$ by Lemma \ref{L:Kilg_Sub}, we can take a finite number of elements $w_1,\dots,w_{k_0}$ of $\mathcal{K}_g(\Gamma)$ such that $\{w_1,\dots,w_{k_0}\}$ is an orthonormal basis of $\mathcal{K}_g(\Gamma)$ with respect to $(g\,\cdot,\cdot)_{L^2(\Gamma)}$.
  Hence, if we define
  \begin{align*}
    v^\varepsilon = v_1^\varepsilon-\sum_{k=1}^{k_0}(gv_1^\varepsilon,w_k)_{L^2(\Gamma)}w_k \in H_{g\sigma}^1(\Gamma,T\Gamma),
  \end{align*}
  then $v$ is orthogonal to $\mathcal{K}_g(\Gamma)$ with respect to $(g\,\cdot,\cdot)_{L^2(\Gamma)}$ and thus $v^\varepsilon\in\mathbb{V}_g$.
  Let us estimate $M_\tau u^\varepsilon-v^\varepsilon$.
  By Lemmas \ref{L:Kilg_Sub} and \ref{L:Lgs_Per}, we have
  \begin{align*}
    (gv_1^\varepsilon,w_k)_{L^2(\Gamma)} = (gM_\tau u^\varepsilon,w_k)_{L^2(\Gamma)}-(g\nabla_\Gamma q^\varepsilon,w_k)_{L^2(\Gamma)} = (gM_\tau u^\varepsilon,w_k)_{L^2(\Gamma)}.
  \end{align*}
  Moreover, since $\mathcal{R}_g|_\Gamma=\mathcal{K}_g(\Gamma)$ in the condition (A3), we may assume that $w_k\in\mathcal{R}_g$ and it is of the form $w_k(x)=a_k\times x+b_k$, $x\in\mathbb{R}^3$ with $a_k,b_k\in\mathbb{R}^3$.
  Then, since $(u^\varepsilon,w_k)_{L^2(\Omega_\varepsilon)}=0$ by $u^\varepsilon\in\mathcal{V}_\varepsilon$ and \eqref{E:Def_Heps}, we can apply \eqref{E:Ave_RPe} and \eqref{E:ATVg_Bo} to find that
  \begin{align*}
    \Bigl|(gv_1^\varepsilon,w_k)_{L^2(\Gamma)}\Bigr| = \Bigl|(gM_\tau u^\varepsilon,w_k)_{L^2(\Gamma)}\Bigr| &\leq c\varepsilon^{1/2}(|a_k|+|b_k|)\|u^\varepsilon\|_{L^2(\Omega_\varepsilon)} \\
    &\leq c\varepsilon^{\sigma/2}(|a_k|+|b_k|)
  \end{align*}
  for each $k=1,\dots,k_0$.
  Hence
  \begin{align*}
    \|v_1^\varepsilon-v^\varepsilon\|_{H^1(\Gamma)} \leq \sum_{k=1}^{k_0}\Bigl|(gv_1^\varepsilon,w_k)_{L^2(\Gamma)}\Bigr|\|w_k\|_{H^1(\Gamma)} \leq c\varepsilon^{\sigma/2}
  \end{align*}
  with a constant $c>0$ independent of $\varepsilon$ (note that $k_0$ and $w_k$ are independent of $\varepsilon$).
  By this inequality and \eqref{Pf_ATV:v1e}, we again get \eqref{E:ATVg_Diff} when $\gamma^0=\gamma^1=0$ and $\mathcal{K}_g(\Gamma)\neq\{0\}$.
\end{proof}

\begin{lemma} \label{L:Weak_Rep}
  Under the assumptions of Lemma \ref{L:Ave_Weak}, let $u^\varepsilon$ be a weak solution to \eqref{E:SNS_CTD} such that $u^\varepsilon\in\mathcal{V}_\varepsilon$.
  When the condition (A1) of Assumption \ref{Asmp_2} is imposed, suppose that $\gamma^0>0$ or $\gamma^1>0$.
  Then there exists a vector field $v^\varepsilon\in\mathbb{V}_g$ such that
  \begin{align} \label{E:WR_Diff}
    \|M_\tau u^\varepsilon-v^\varepsilon\|_{H^1(\Gamma)} \leq c\varepsilon,
  \end{align}
  where $c>0$ is a constant independent of $\varepsilon$, $u^\varepsilon$, and $v^\varepsilon$.
  Moreover,
  \begin{align} \label{E:Weak_Rep}
    a_g(v^\varepsilon,\eta)+b_g(v^\varepsilon,v^\varepsilon,\eta) = (gM_\tau\mathbb{P}_\varepsilon f^\varepsilon,\eta)_{L^2(\Gamma)}+R_\varepsilon^2(\eta)
  \end{align}
  for all $\eta\in H_{g\sigma}^1(\Gamma,T\Gamma)$, where $R_\varepsilon^2(\eta)$ is a residual term for which \eqref{E:AW_Res} holds.
\end{lemma}

\begin{proof}
  Since $u^\varepsilon\in\mathcal{V}_\varepsilon$ and \eqref{E:ATVg_Bo} is valid with $\sigma=2$ by \eqref{E:AW_H1H2}, there exists a vector field $v^\varepsilon\in\mathbb{V}_g$ such that \eqref{E:WR_Diff} holds by Lemma \ref{L:AveT_Vg}.
  Moreover,
  \begin{align} \label{Pf_WR:Bound}
    \begin{aligned}
      \|M_\tau u^\varepsilon\|_{H^1(\Gamma)} &\leq c\varepsilon^{-1/2}\|u^\varepsilon\|_{H^1(\Omega_\varepsilon)} \leq c, \\
      \|v^\varepsilon\|_{H^1(\Gamma)} &\leq \|M_\tau u^\varepsilon\|_{H^1(\Gamma)}+\|M_\tau u^\varepsilon-v^\varepsilon\|_{H^1(\Gamma)} \leq c
    \end{aligned}
  \end{align}
  with a constant $c>0$ independent of $\varepsilon$ by \eqref{E:Ave_Hk} with $k=1$, \eqref{E:AW_H1H2}, and \eqref{E:WR_Diff}.

  Let $\eta\in H_{g\sigma}^1(\Gamma,T\Gamma)$.
  We see by \eqref{E:Ave_Weak} that, if we set
  \begin{align*}
    J_1 &= a_g(M_\tau u^\varepsilon,\eta)-a_g(v^\varepsilon,\eta), \\
    J_2 &= b_g(M_\tau u^\varepsilon,M_\tau u^\varepsilon,\eta)-b_g(v^\varepsilon,v^\varepsilon,\eta),
  \end{align*}
  then $v^\varepsilon$ satisfies \eqref{E:Weak_Rep} with $R_\varepsilon^2(\eta)=R_\varepsilon^1(\eta)-J_1-J_2$, where $R_\varepsilon^1(\eta)$ satisfies \eqref{E:AW_Res}.
  Moreover, it follows from \eqref{E:BLS_Bou} and \eqref{E:WR_Diff} that
  \begin{align*}
    |J_1| = |a_g(M_\tau u^\varepsilon-v^\varepsilon,\eta)| \leq c\|M_\tau u^\varepsilon-v^\varepsilon\|_{H^1(\Gamma)}\|\eta\|_{H^1(\Gamma)} \leq c\varepsilon\|\eta\|_{H^1(\Gamma)}.
  \end{align*}
  Also, since $J_2=b_g(M_\tau u^\varepsilon-v^\varepsilon,M_\tau u^\varepsilon,\eta)+b_g(v^\varepsilon,M_\tau u^\varepsilon-v^\varepsilon,\eta)$, we have
  \begin{align*}
    |J_2| \leq c\Bigl(\|M_\tau u^\varepsilon\|_{H^1(\Gamma)}+\|v^\varepsilon\|_{H^1(\Gamma)}\Bigr)\|M_\tau u^\varepsilon-v^\varepsilon\|_{H^1(\Gamma)}\|\eta\|_{H^1(\Gamma)} \leq c\varepsilon\|\eta\|_{H^1(\Gamma)}
  \end{align*}
  by \eqref{E:TLS_Bou}, \eqref{E:WR_Diff}, and \eqref{Pf_WR:Bound}.
  Hence $R_\varepsilon^2(\eta)$ satisfies \eqref{E:AW_Res} by $\varepsilon\leq\varepsilon^{\alpha/4}$.
\end{proof}

\subsection{Difference estimate on the surface} \label{SS:Pf_TErr}
Now we are ready to prove Theorem \ref{T:Error}.

\begin{proof}[Proof of Theorem \ref{T:Error}]
  Let $u^\varepsilon$ and $v$ be weak solutions to \eqref{E:SNS_CTD} and \eqref{E:Lim_Eq}, respectively.
  Suppose that the assumptions of Theorem \ref{T:Error} are satisfied.
  Then there exists a vector field $v^\varepsilon\in\mathbb{V}_g$ satisfying \eqref{E:WR_Diff} and \eqref{E:Weak_Rep} by Lemma \ref{L:Weak_Rep}.
  For $\eta\in H_{g\sigma}^1(\Gamma,T\Gamma)$, we subtract \eqref{E:Weak_Rep} from \eqref{E:Lim_Weak} and use \eqref{E:HinS_L2} and \eqref{E:HinS_Mul} to get
  \begin{align*}
    a_g(v-v^\varepsilon,\eta)+b_g(v-v^\varepsilon,v,\eta)+b_g(v^\varepsilon,v-v^\varepsilon,\eta) = [f-M_\tau\mathbb{P}_\varepsilon f^\varepsilon,g\eta]_{T\Gamma}-R_\varepsilon^2(\eta).
  \end{align*}
  Let $\eta=v-v^\varepsilon$ in this equality.
  Then, since $\gamma^0>0$ or $\gamma^1>0$ or $v-v^\varepsilon\in\mathbb{V}_g$ satisfies \eqref{E:KoG_Per} by the assumptions of Theorem \ref{T:Error}, we can use \eqref{E:BLS_Cor} to get
  \begin{align*}
    \|v-v^\varepsilon\|_{H^1(\Gamma)}^2 \leq c_aa_g(v-v^\varepsilon,v-v^\varepsilon).
  \end{align*}
  Moreover, we observe by \eqref{E:TLS_Bou} and \eqref{E:TLS_Asym} that
  \begin{align*}
    |b_g(v-v^\varepsilon,v,v-v^\varepsilon)| \leq c_b\|v\|_{H^1(\Gamma)}\|v-v^\varepsilon\|_{H^1(\Gamma)}^2, \quad b_g(v^\varepsilon,v-v^\varepsilon,v-v^\varepsilon) = 0,
  \end{align*}
  and by $g\in C^1(\Gamma)$ that, for any fixed $\sigma\in(0,1)$,
  \begin{align*}
    \Bigl|[f-M_\tau\mathbb{P}_\varepsilon f^\varepsilon,g(v-v^\varepsilon)]_{T\Gamma}\Bigr| &\leq \|f-M_\tau\mathbb{P}_\varepsilon f^\varepsilon\|_{H^{-1}(\Gamma,T\Gamma)}\|g(v-v^\varepsilon)\|_{H^1(\Gamma)} \\
    &\leq c\|f-M_\tau\mathbb{P}_\varepsilon f^\varepsilon\|_{H^{-1}(\Gamma,T\Gamma)}\|v-v^\varepsilon\|_{H^1(\Gamma)} \\
    &\leq \frac{\sigma}{2c_a}\|v-v^\varepsilon\|_{H^1(\Gamma)}^2+\frac{c}{\sigma}\|f-M_\tau\mathbb{P}_\varepsilon f^\varepsilon\|_{H^{-1}(\Gamma,T\Gamma)}^2.
  \end{align*}
  Since \eqref{E:AW_Res} holds for $R_\varepsilon^2(\eta)$ by Lemma \ref{L:Weak_Rep}, we also have
  \begin{align*}
    |R_\varepsilon^2(v-v^\varepsilon)| \leq c\delta(\varepsilon)\|v-v^\varepsilon\|_{H^1(\Gamma)} \leq \frac{\sigma}{2c_a}\|v-v^\varepsilon\|_{H^1(\Gamma)}^2+\frac{c}{\sigma}\delta(\varepsilon)^2,
  \end{align*}
  where $\delta(\varepsilon)=\varepsilon^{\alpha/4}+\sum_{i=0,1}|\varepsilon^{-1}\gamma_\varepsilon^i-\gamma^i|$.
  Hence we get
  \begin{align*}
    \|v-v^\varepsilon\|_{H^1(\Gamma)}^2 &\leq \Bigl(\sigma+c_ac_b\|v\|_{H^1(\Gamma)}\Bigr)\|v-v^\varepsilon\|_{H^1(\Gamma)}^2\\
    &\qquad +\frac{c}{\sigma}\Bigl(\delta(\varepsilon)^2+\|f-M_\tau\mathbb{P}_\varepsilon f^\varepsilon\|_{H^{-1}(\Gamma,T\Gamma)}^2\Bigr),
  \end{align*}
  where $c>0$ is a constant independent of $\varepsilon$ and $\sigma$.
  Thus, if we set
  \begin{align*}
    \rho = \rho(\sigma) = \frac{1-2\sigma}{c_ac_b} > 0
  \end{align*}
  for any fixed $\sigma\in(0,1)$ and assume that $\|v\|_{H^1(\Gamma)}\leq\rho$, then it follows that
  \begin{align*}
    \|v-v^\varepsilon\|_{H^1(\Gamma)}^2 \leq \frac{c}{\sigma^2}\Bigl(\delta(\varepsilon)^2+\|f-M_\tau\mathbb{P}_\varepsilon f^\varepsilon\|_{H^{-1}(\Gamma,T\Gamma)}^2\Bigr).
  \end{align*}
  Therefore, we obtain \eqref{E:Error} by this inequality and \eqref{E:WR_Diff} (note that $\varepsilon\leq\varepsilon^{\alpha/4}$).
\end{proof}

\begin{remark} \label{R:Pf_TErr}
  Let $\rho_u=(c_ac_b)^{-1}$ be the constant given in Proposition \ref{P:LW_Uniq}.
  For any fixed $\sigma\in(0,1)$, the constant $\rho=\rho(\sigma)$ given above satisfies $\rho<\rho_u$.
  Hence the condition $\|v\|_{H^1(\Gamma)}\leq\rho$ implies the uniqueness of a weak solution to \eqref{E:Lim_Eq} in the class $\mathbb{V}_g$ by Proposition \ref{P:LW_Uniq}.
  Conversely, if a weak solution $v\in\mathbb{V}_g$ to \eqref{E:Lim_Eq} satisfies $\|v\|_{H^1(\Gamma)}<\rho_u$, then $\|v\|_{H^1(\Gamma)}\leq\rho(\sigma)$ for some $\sigma\in(0,1)$ and Theorem \ref{T:Error} is applicable to $v$.
\end{remark}

\begin{remark} \label{R:Er_Full}
  Under the assumptions of Theorem \ref{T:Error}, we can apply Proposition \ref{P:NCW_AN} with $\beta=1$ to $u^\varepsilon$.
  Then, noting that $\alpha\in(0,1]$, we have
  \begin{align*}
    \|Mu^\varepsilon\cdot n\|_{H^1(\Gamma)} \leq c\varepsilon^{\alpha/2}.
  \end{align*}
  By this inequality and $Mu^\varepsilon=M_\tau u^\varepsilon+(Mu^\varepsilon\cdot n)n$ on $\Gamma$, we see that the difference estimate \eqref{E:Error} is still valid if we replace $M_\tau u^\varepsilon$ by the average $Mu^\varepsilon$ itself.
\end{remark}

\subsection{Difference estimate in the thin domain} \label{SS:Pf_ErCE}
Let us show Theorem \ref{T:Er_CE}.

\begin{proof}[Proof of Theorem \ref{T:Er_CE}]
  The proof is the same as in the nonstationary case \cite[Theorems 7.29 and 7.30]{Miu20_03}, but we give it here for the reader's convenience.

  Under the assumptions of Theorem \ref{T:Error}, let $u^\varepsilon$ and $v$ be weak solutions to \eqref{E:SNS_CTD} and \eqref{E:Lim_Eq}, respectively.
  Then $u^\varepsilon\in D(A_\varepsilon)$ by Lemma \ref{L:Ave_Weak} and $u^\varepsilon$ satisfies \eqref{E:Slip_Bo} on $\Gamma_\varepsilon$.
  We have
  \begin{align*}
    \Bigl\|u^\varepsilon-\overline{M_\tau u^\varepsilon}\Bigr\|_{L^2(\Omega_\varepsilon)}+\Bigl\|\overline{P}\nabla u^\varepsilon-\overline{\nabla_\Gamma M_\tau u^\varepsilon}\Bigr\|_{L^2(\Omega_\varepsilon)} \leq c\varepsilon\|u^\varepsilon\|_{H^2(\Omega_\varepsilon)} \leq c\varepsilon^{1/2+\alpha/2}
  \end{align*}
  by \eqref{E:Ave_Diff}, \eqref{E:AvDi_H1}, and \eqref{E:AW_H1H2}.
  Moreover, it follows from \eqref{E:CE_L2CTD} that
  \begin{align} \label{Pf_TEC:Mtuv}
    \Bigl\|\overline{M_\tau u^\varepsilon}-\bar{v}\Bigr\|_{L^2(\Omega_\varepsilon)}+\Bigl\|\overline{\nabla_\Gamma M_\tau u^\varepsilon}-\overline{\nabla_\Gamma v}\Bigr\|_{L^2(\Omega_\varepsilon)} \leq c\varepsilon^{1/2}\|M_\tau u^\varepsilon-v\|_{H^1(\Gamma)}.
  \end{align}
  By these inequalities, \eqref{E:Error}, and $\varepsilon^{\alpha/2}\leq\varepsilon^{\alpha/4}$, we find that
  \begin{multline} \label{Pf_TEC:uvL2}
    \varepsilon^{-1/2}\left(\|u^\varepsilon-\bar{v}\|_{L^2(\Omega_\varepsilon)}+\Bigl\|\overline{P}\nabla u^\varepsilon-\overline{\nabla_\Gamma v}\Bigr\|_{L^2(\Omega_\varepsilon)}\right) \\
    \leq c\Bigl(\delta(\varepsilon)+\|M_\tau\mathbb{P}_\varepsilon f^\varepsilon-f\|_{H^{-1}(\Gamma,T\Gamma)}\Bigr),
  \end{multline}
  where $\delta(\varepsilon)=\varepsilon^{\alpha/4}+\sum_{i=0,1}|\varepsilon^{-1}\gamma_\varepsilon^i-\gamma^i|$.
  Next we have
  \begin{align*}
    \Bigl\|\overline{P}\partial_nu^\varepsilon+\overline{W}u^\varepsilon\Bigr\|_{L^2(\Omega_\varepsilon)}+\left\|\partial_nu^\varepsilon\cdot\bar{n}-\frac{1}{\bar{g}}\overline{M_\tau u^\varepsilon}\cdot\overline{\nabla_\Gamma g}\right\|_{L^2(\Omega_\varepsilon)} \leq c\varepsilon\|u^\varepsilon\|_{H^2(\Omega_\varepsilon)} \leq c\varepsilon^{1/2+\alpha/2}
  \end{align*}
  by \eqref{E:Pdn_CTD}, \eqref{E:ndn_CTD}, and \eqref{E:AW_H1H2}.
  Noting that $W$ and $\nabla_\Gamma g$ are bounded on $\Gamma$ and \eqref{E:G_Inf} holds, we deduce from this inequality, \eqref{E:Error}, \eqref{Pf_TEC:Mtuv}, \eqref{Pf_TEC:uvL2}, and $\varepsilon^{\alpha/2}\leq\varepsilon^{\alpha/4}$ that
  \begin{multline} \label{Pf_TEC:dnuv}
    \varepsilon^{-1/2}\left(\Bigl\|\overline{P}\partial_nu^\varepsilon+\overline{Wv}\Bigr\|_{L^2(\Omega_\varepsilon)}+\left\|\partial_nu^\varepsilon\cdot\bar{n}-\frac{1}{\bar{g}}\bar{v}\cdot\overline{\nabla_\Gamma g}\right\|_{L^2(\Omega_\varepsilon)}\right) \\
    \leq c\Bigl(\delta(\varepsilon)+\|M_\tau\mathbb{P}_\varepsilon f^\varepsilon-f\|_{H^{-1}(\Gamma,T\Gamma)}\Bigr).
  \end{multline}
  Therefore, setting $V=-Wv+g^{-1}(v\cdot\nabla_\Gamma g)n$ on $\Gamma$ and noting that
  \begin{align*}
    \partial_nu^\varepsilon-\overline{V} = \Bigl(\overline{P}\partial_nu^\varepsilon+\overline{Wv}\Bigr)+\left(\partial_nu^\varepsilon\cdot\bar{n}-\frac{1}{\bar{g}}\bar{v}\cdot\overline{\nabla_\Gamma g}\right)\bar{n} \quad\text{in}\quad \Omega_\varepsilon,
  \end{align*}
  we obtain \eqref{E:Er_CE} by \eqref{Pf_TEC:uvL2} and \eqref{Pf_TEC:dnuv}.
\end{proof}

\section{Proof of Lemma \ref{L:Sob_CTD}} \label{S:Pf_Sob}
In this section we prove Lemma \ref{L:Sob_CTD}.
First suppose that \eqref{E:Sob_L6} holds.
Then
\begin{align*}
  \|\varphi\|_{L^2(\Omega_\varepsilon)} \leq \|\varphi\|_{H^1(\Omega_\varepsilon)}, \quad \|\varphi\|_{L^6(\Omega_\varepsilon)} \leq c\varepsilon^{-1/3}\|\varphi\|_{H^1(\Omega_\varepsilon)}
\end{align*}
for $\varphi\in H^1(\Omega_\varepsilon)$ and thus we have \eqref{E:Sob_Lp} by an interpolation.

Let us show \eqref{E:Sob_L6}.
We follow the idea of the proof of \cite[Lemma 4.4]{HoaSel10} in the case of a flat thin domain and use the anisotropic Sobolev inequality
\begin{align} \label{E:Ani_Sob}
  \|\Phi\|_{L^6(Q)} \leq c\prod_{i=1}^3\Bigl(\|\Phi\|_{L^2(Q)}+\|\partial_i\Phi\|_{L^2(Q)}\Bigr)^{1/3}
\end{align}
for $Q=(0,1)^3$ and $\Phi\in H^1(Q)$ (see \cite[Proposition 2.3]{TemZia96}).

Since $\Gamma$ is closed and of class $C^5$, we can take a finite number of bounded open subsets $U_k$ of $\mathbb{R}^2$ and $C^5$ local parametrizations $\mu_k\colon U_k\to\Gamma$ of $\Gamma$ with $k=1,\dots,k_0$ such that $\{\mu_k(U_k)\}_{k=1}^{k_0}$ is an open covering of $\Gamma$.
Let $\{\eta_k\}_{k=1}^{k_0}$ be a partition of unity on $\Gamma$ subordinate to $\{\mu_k(U_k)\}_{k=1}^{k_0}$.
We may assume that $\eta_k$ is supported in $\mu_k(\mathcal{K}_k)$ with some compact subset $\mathcal{K}_k$ of $U_k$ for each $k=1,\dots,k_0$.
Let $Q_k=U_k\times(0,1)$ and
\begin{align*}
  \zeta_k(s) = \mu_k(s')+\varepsilon\{(1-s_3)g_0(\mu_k(s'))+s_3g_1(\mu_k(s'))\}n(\mu_k(s'))
\end{align*}
for $s=(s',s_3)\in Q_k$.
Also, let $\bar{\eta}_k=\eta_k\circ\pi$ be the constant extension of $\eta_k$ in the normal direction of $\Gamma$.
Then $\{\zeta_k(Q_k)\}_{k=1}^{k_0}$ is an open covering of $\Omega_\varepsilon$ and $\{\bar{\eta}_k\}_{k=1}^{k_0}$ is a partition of unity on $\Omega_\varepsilon$ subordinate to $\{\zeta_k(Q_k)\}_{k=1}^{k_0}$.
Moreover, $\partial_n\bar{\eta}_k=\bar{n}\cdot\nabla\bar{\eta}_k=0$ in $\Omega_\varepsilon$.

Let $\varphi\in H^1(\Omega_\varepsilon)$ and $\varphi_k=\bar{\eta}_k\varphi$ on $\Omega_\varepsilon$ for $k=1,\dots,k_0$.
Then, since $\varphi_k$ is supported in $\zeta(\mathcal{K}_k\times(0,1))\subset\zeta_k(Q_k)$ and $\partial_n\varphi_k=\bar{\eta}_k\partial_n\varphi$ in $\Omega_\varepsilon$, it is sufficient for \eqref{E:Sob_L6} to show that
\begin{align} \label{E:Sob_Loc}
  \|\varphi_k\|_{L^6(\zeta_k(Q_k))} \leq c\varepsilon^{-1/3}\|\varphi_k\|_{H^1(\zeta_k(Q_k))}^{2/3}\Bigl(\|\varphi_k\|_{L^2(\zeta_k(Q_k))}+\varepsilon\|\partial_n\varphi_k\|_{L^2(\zeta_k(Q_k))}\Bigr)^{1/3}
\end{align}
for all $k=1,\dots,k_0$.
From now on, we fix and suppress the index $k$.
Hence $\varphi$ is supported in $\zeta(\mathcal{K}\times(0,1))\subset\zeta(Q)$.
Moreover, we may assume that $U=(0,1)^2$ and thus $Q=(0,1)^3$ by taking $U$ sufficiently small and scaling it.
We also write $c$ for a general positive constant independent of $\varepsilon$.
The local parametrization of $\Omega_\varepsilon$ on $Q$ is of the form
\begin{align*}
  \zeta(s) = \mu(s')+h_\varepsilon(s)n(\mu(s')), \quad s=(s',s_3)\in Q,
\end{align*}
where $\mu\colon U\to\Gamma$ is a $C^5$ local parametrization of $\Gamma$ and
\begin{align*}
  h_\varepsilon(s) = \varepsilon\{(1-s_3)g_0(\mu(s'))+s_3g_1(\mu(s'))\}.
\end{align*}
Then, since $g_0$, $g_1$, and $n$ are of class $C^4$ on $\Gamma$ and $\mathcal{K}$ is compact in $U$, we have
\begin{align} \label{E:Gr_Ze}
  |\nabla_s\zeta(s)| \leq c, \quad s\in\mathcal{K}\times(0,1),
\end{align}
where $\nabla_s\zeta$ is the gradient matrix of $\zeta$.
Moreover, it is shown in the proof of \cite[Lemma 4.3]{Miu21_02} (see (B.11) in \cite{Miu21_02}) that
\begin{align*}
  \det\nabla_s\zeta(s) = \varepsilon g(\mu(s'))J(\mu(s'),h_\varepsilon(s))\sqrt{\det\theta(s')}, \quad s\in Q,
\end{align*}
where $J$ is given by \eqref{E:Def_Ja} and $\theta=(\theta_{ij})_{i,j}$ is the Riemannian metric of $\Gamma$ locally given by $\theta_{ij}=\partial_{s_i}\mu\cdot\partial_{s_j}\mu$ on $U$ for $i,j=1,2$.
Then, noting that $\det\theta$ is continuous and positive on $U$ and $\mathcal{K}$ is compact in $U$, we see by the above equality, \eqref{E:G_Inf}, and \eqref{E:Jacob} that
\begin{align} \label{E:Det_Ze}
  c^{-1}\varepsilon \leq \det\nabla_s\zeta(s) \leq c\varepsilon, \quad s\in \mathcal{K}\times(0,1).
\end{align}
Now let $\Phi=\varphi\circ\zeta$ on $Q$.
Then since $\varphi$ is supported in $\zeta(\mathcal{K}\times(0,1))$ and
\begin{align*}
  \partial_{s_i}\Phi(s) &= \partial_{s_i}\zeta(s)\cdot\nabla\varphi(\zeta(s)), \quad i=1,2, \\
  \partial_{s_3}\Phi(s) &= \varepsilon g(\mu(s'))\partial_n\varphi(\zeta(s))
\end{align*}
for all $s\in Q$, it follows from \eqref{E:G_Inf}, \eqref{E:Gr_Ze}, and \eqref{E:Det_Ze} that
\begin{align} \label{E:Phi_H1}
  \begin{aligned}
    \|\Phi\|_{L^2(Q)} &\leq c\varepsilon^{-1/2}\|\varphi\|_{L^2(\zeta(Q))}, \\
    \|\partial_{s_i}\Phi\|_{L^2(Q)} &\leq c\varepsilon^{-1/2}\|\nabla\varphi\|_{L^2(\zeta(Q))}, \quad i=1,2, \\
    \|\partial_{s_3}\Phi\|_{L^2(Q)} & \leq c\varepsilon^{1/2}\|\partial_n\varphi\|_{L^2(\zeta(Q))}
  \end{aligned}
\end{align}
and thus $\Phi\in H^1(Q)$ and \eqref{E:Ani_Sob} holds for $\Phi$.
Moreover, since \eqref{E:Det_Ze} holds and $\varphi$ is supported in $\zeta(\mathcal{K}\times(0,1))$, it follows that
\begin{align} \label{E:Phi_L6}
  c\varepsilon^{-1/6}\|\varphi\|_{L^6(\zeta(Q))} \leq \|\Phi\|_{L^6(Q)}.
\end{align}
Hence we obtain \eqref{E:Sob_Loc} by \eqref{E:Ani_Sob}, \eqref{E:Phi_H1}, and \eqref{E:Phi_L6}.

\section*{Acknowledgments}
The work of the author was supported by JSPS KAKENHI Grant Number 23K12993.

\bibliographystyle{abbrv}
\bibliography{SNS_CTD_Ref}

\end{document}